\documentclass[10pt,letterpaper]{amsart}

\usepackage{amsfonts}
\usepackage{amsmath}
\usepackage{amsrefs}
\usepackage{amssymb}
\usepackage{amsthm}
\usepackage{array}
\usepackage{bm}
\usepackage{dcpic,pictexwd}
\usepackage{fancyhdr}
\usepackage{indentfirst}
\usepackage{ifthen}
\usepackage{accents}
\numberwithin{equation}{section}
\newtheorem{Theorem}{Theorem}[section]

\newtheorem{Lemma}[Theorem]{Lemma}
\newtheorem{Corollary}[Theorem]{Corollary}

\theoremstyle{definition}
\newtheorem{Example}[Theorem]{Example}
\newtheorem{Remark}[Theorem]{Remark}
\newfont{\deffont}{cmbxti10}
\newfont{\german}{eufm10}
\newfont{\mymath}{cmr12}

\newcommand\lieg{\mathfrak{g}}
\newcommand\lieh{\mathfrak{h}}
\newcommand\liek{\mathfrak{k}}
\newcommand\liel{\mathfrak{l}}

\newcommand\tomega{\tilde \omega}

\newcommand\hook{\mathbin{\raise0.5pt\hbox{\hbox{{\vbox{\hrule height.4pt width6pt depth0pt}}}\vrule height6pt width.4pt depth0pt}\,}}

\newcommand\cTM{T^*\kern-2ptM}

\newcommand\Ad{\text{\rm Ad}}
\newcommand{\ad}{\text{\rm ad}}

\newcommand\omegaX{\omega_{\kern -1 pt X}}

\newcommand\ann{\text{\rm ann}}
\newcommand\CalPf{\mathcal{P}\text{\it \kern -.3pt f}}

\newcommand\der{\rm{Der}}

\newcommand\TM{T\kern -2pt M}

\newcommand\real{\text{\bf R}}
\newcommand\mycap{\hbox{\ $\rlap{\kern -.3pt $\cap$}\raise.8pt\hbox{$\scriptstyle+$}$\ } }

\newcommand\Ao{{\kern-2.3pt}\stackrel{\scriptscriptstyle o}{A}{}{\kern-2.3pt}}
\newcommand\Bo{{\kern-2.3pt}\stackrel{\scriptscriptstyle o}{B}{}{\kern-2.3pt}}
\newcommand\Ko{{\kern-2.3pt}\stackrel{\scriptscriptstyle o}{K}{}{\kern-2.3pt}}
\newcommand\Co{{\kern-2.3pt}\stackrel{\scriptscriptstyle o}{C}{}{\kern-2.3pt}}
\newcommand\Qo{{\kern-2.3pt}\stackrel{\scriptscriptstyle o}{Q}{}{\kern-2.3pt}}
\newcommand\Mo{{\kern-2.3pt}\stackrel{\scriptscriptstyle o}{M}{}{\kern-2.3pt}}

\newcommand\Xo{{\stackrel{\scriptscriptstyle o}{X}}{\kern-1.3pt}}
\newcommand\yo{{\stackrel{\scriptscriptstyle o}{y}}{\kern-1.3pt}}

\DeclareMathOperator{\spn}{span}

\newboolean{proofmode}

\newcommand{\StTag}[1]{ \label{st:#1}
\ifthenelse{\boolean{proofmode}}{\ \marginpar{\quad\scriptsize st:#1} }{}      }

\newcommand{\EqTag}[1]{
\ifthenelse{\boolean{proofmode}}
{ {\label{eq:#1}}
  \stepcounter{equation}
  \tag{\theequation \rlap{\kern 23 pt{\scriptsize eq:#1}}}
}
{\label{eq:#1}}
 }

\newcommand{\EqRef}[1]{\eqref{eq:#1}}
\newcommand{\StRef}[1]{\ref{st:#1}}
\newcolumntype{C}{>\scriptstyle>{$}c <{$} }
\newcolumntype{L}{>\scriptstyle >{$} l <{$} }

\newcommand\CalI{\mathcal{I}}
\newcommand\CalL{\mathcal{L}}
\newcommand\CalN{\mathcal{N}}

\newcommand\sd{\mathbin{ \raise0.0pt\hbox{ \vrule height5pt width.4pt depth0pt}\!\times}}

\newcommand\barC{
	\hbox{\kern 2.3 true pt
	\vbox{\hrule width 6.5  true pt height .3 true pt \kern .9 true pt
	\hbox{\kern -0.8 true pt $C$}}}}

\newcommand\barM{
	\hbox{\kern 2.3 true pt
	\vbox{\hrule width 8.5  true pt height .3 true pt \kern .9 true pt
	\hbox{\kern -2.3 true pt $M$}}}}
\newcommand\barU{
	\hbox{\kern .8 true pt
	\vbox{\hrule width 6.5  true pt height .3 true pt \kern .9 true pt
	\hbox{\kern -.8 true pt $U$}}}}

\newcommand\barCalI{
	\hbox{\kern 4.3 true pt
	\vbox{\hrule width 6.5  true pt height .3 true pt \kern .9 true pt
	\hbox{ \kern -4.3 true pt  $\CalI$}}}}
\newcommand\barXi{
	\hbox{\kern 1 true pt
	\vbox{\hrule width 6.5  true pt height .3 true pt \kern .9 true pt
	\hbox{\kern 1 true pt $\Xi$}}}}

\newcommand\Largehat{\smash{\raise -7.5 pt \hbox{\rm\Large\^{}}}}
\newcommand\LARGEhat{\smash{\raise -9.5 pt \hbox{\rm\LARGE\^{}}}}
\newcommand\hugehat{\smash{\raise -7.5 pt \hbox{\rm\huge\^{}}}}
\newcommand\Hugehat{\smash{\raise -7.5 pt \hbox{\rm\Huge\^{}}}}

\newcommand\Largecheck{\smash{\raise -7.5 pt \hbox{\rm\Large\v{}}}}
\newcommand\LARGEcheck{\smash{\raise -9 pt \hbox{\rm\LARGE\v{}}}}
\newcommand\hugecheck{\smash{\raise -7.5 pt \hbox{\rm\huge\v{}}}}
\newcommand\Hugecheck{\smash{\raise -7.5 pt \hbox{\rm\Huge\v{}}}}

\newcommand\homega{\accentset{\Largehat}{\omega}}
\newcommand\hbfomega{\accentset{\Largehat}{\boldsymbol \omega}}

\newcommand\htau{\accentset{\Largehat}{\tau}}

\newcommand\comega{\accentset{\Largecheck}{\omega}}

\newcommand\bfe{\boldsymbol{e}}
\newcommand\be{\boldsymbol{e}}

\newcommand\rme{{\rm e}}

\newcommand\bfy{{y}}

\newcommand\bfGamma{\boldsymbol{\Gamma}}

\newcommand\bfomega{{\boldsymbol{\omega}}}
\newcommand\bftau{{\boldsymbol{\tau}}}
\newcommand\bftheta{{\boldsymbol{\theta}}}
\newcommand\bfttheta{\tilde{\boldsymbol{\theta}}}

\newcommand\bfOmega{{\boldsymbol{\Omega}}}

\newcommand\bfomegaX{{\boldsymbol{\omega_{\kern -1 pt X}}}}

\newcommand\bfhomega{\boldsymbol{{\hat \omega}}}

\newcommand\bfcomega{\boldsymbol{{\check \omega}}}

\newcommand\Romega{{ \raise 1pt \hbox{$\scriptstyle {\boldsymbol{\omega}}$}}}
\newcommand\Lomega{{ \lower 1pt \hbox{$\scriptstyle {\boldsymbol{\omega}}$}}}

\newcommand\Rsigma{{ \raise 1pt \hbox{$\scriptstyle \sigma$}}}

\newcommand\Reta{{ \raise 1pt \hbox{$\scriptstyle \eta$}}}

\newcommand\vecthsigma{\partial_{{\displaystyle \hat {\raise 1.3pt \hbox{$\scriptstyle \sigma$}}}^a}}
\newcommand\vectcsigma{\partial_{{\displaystyle \check {\raise 1.3pt \hbox{$\scriptstyle \sigma$}}}^\alpha}}

\setboolean{proofmode}{false}

\def\Obj(#1, #2)[#3]#4{\obj(#1, #2)[#3]{#4}}
\def\beginDC#1[#2]{\begindc{#1}[#2]}

\begin{document}

\title{On the Construction of Simply Connected Solvable Lie Groups}

\author{M.E. Fels}
\address{Department of Mathematics and Statistics, Utah State University, Logan Utah, 84322}
\curraddr{Department of Mathematics and Statistics, Utah State University, Logan Utah, 84322}
\email{Mark.Fels@usu.edu}


\date{\today}



\begin{abstract}  Let $\omega_\mathfrak{g}$ be a Lie algebra valued differential $1$-form  on a manifold $M$ satisfying the structure equations $d \omega_\mathfrak{g} + \frac{1}{2} \omega_\mathfrak{g}\wedge \omega_\lieg=0$ where $\mathfrak{g}$ is solvable. We show that the problem of finding a smooth map $\rho:M\to G$, where $G$ is an $n$-dimensional solvable Lie group with Lie algebra $\mathfrak{g}$ and left invariant Maurer-Cartan form $\tau$, such that $\rho^* \tau= \omega_\mathfrak{g}$ can be solved by quadratures and the matrix exponential.  In the process we give a closed form formula for the vector fields in Lie's third theorem for solvable Lie algebras. A further application produces the multiplication map for a simply connected  $n$-dimensional  solvable Lie group using only the matrix exponential and $n$ quadratures. Applications to finding first integrals for completely integrable Pfaffian systems with solvable symmetry algebras are also given.
\end{abstract}

\pagestyle{plain}
\maketitle

\section{Introduction}

Let $ M$ be an $m$-dimensional manifold and let $\omega^i$, $i=1,\ldots,n$ be a collection of  $n$ differential forms
on $M$ satisfying
\begin{equation}
d\omega^i +\frac{1}{2} C^i_{jk} \omega^j \wedge \omega^k =0
\EqTag{FSE}
\end{equation}
where $C^i_{jk}$ are the structure constants of an $n$-dimensional real Lie algebra $\lieg$. Given $x\in M$ there exists an open set $U$ about $x$ and a smooth map $\rho:U \to G$, where $G$ is a Lie group with a basis of left invariant forms $\tau^i$ on $G$ satisfying \EqRef{FSE} on $G$, such that
\begin{equation}
\omega^i = \rho^* \tau^i.
\EqTag{MF1}
\end{equation}
If $M$ is simply connected, then $U$ can be taken to be all of $M$. These facts and many others related to equations \EqRef{FSE} and \EqRef{MF1} are proved in Chapter 3 of \cite{sharpe:1997a}.

Equation \EqRef{FSE} occurs in a number of applications such as the integration of Pfaffian systems of finite type with symmetry (see  \cite{doubrov:2000a}, \cite{fels:2007a} and \cite{lychagin:1991a}), the reconstruction problem for quotients of Pfaffian systems (see \cite{anderson-fels:2005a} and \cite{fels:2007a}), and the method of moving frames \cite{griffiths:1974a}. Equation \EqRef{FSE} also occurs in Lie's third theorem. One version of this theorem (see Section \StRef{SLieT}) states that given a Lie algebra $\lieg$ with structure constants $C^i_{jk}$, there exists $n$ point-wise linearly independent $1$-forms $\omega^i$ on $\real^n$ satisfying equation \EqRef{FSE}.



\smallskip

This article addresses how quadrature and matrix algebra can be used to find the map $\rho:M \to G$ in equation \EqRef{MF1} when $\lieg$ is a solvable Lie algebra. The solution to this problem leads to an algorithm that can be used to explicitly construct simply connected solvable Lie groups from their Lie algebra using matrix exponentiation and quadrature. By quadrature we mean the following (see \cite{lychagin:1991a}). Given a closed differential $1$-form $\omega$ on a simply connected manifold $M$, then $\omega  = df $ for a smooth function $f:M \to \real$ where $f$ can be determined from $\omega$ by integration (quadrature) along a path from a fixed point (the Poincare Lemma for $1$-forms).

Our starting point is Theorem \StRef{reduce} in Section \StRef{Sreduce} which shows that if a Lie algebra $\lieg$ admits a codimension one ideal $\liek \subset \lieg$, then the structure equations \EqRef{FSE} can be reduced to $\liek \times \real$ using one quadrature and the matrix exponential. A simple induction argument can then be applied when $\lieg$ is a solvable Lie algebra which reduces the structure equations \EqRef{FSE}  to $d \tilde \omega^i = 0$ (Corollary \StRef{C2}).

In Section \StRef{Sfi} we show how the results in Section \StRef{Sreduce} can be used to construct first integrals for a completely integrable Pfaffian system with a solvable symmetry algebra using quadrature and matrix exponentiation.

In Section \StRef{SLieT} we reverse the argument from Corollary \StRef{C2} in Section \StRef{Sreduce} which results in a proof of Lie's Third theorem for vector fields by giving an {\it explicit closed form} for the vector fields,  see Theorem \StRef{PLT} and Corollary \StRef{CLV}. In comparison the proof of the vector field version of Lie's third theorem given by Cartan, which is valid for any Lie algebra, uses ODE existence theorems (see  \cite{flanders:1989a}) and does not give closed form formulas.

In Section \StRef{SLie3G} our main goal is to show that the multiplication map $\mu:G \times G\to G$  for any simply connected solvable Lie group $G$ can be constructed by $n$ quadratures and the matrix exponential (Theorem \StRef{L3M}). This is done in 3 steps. 

The first step is given in Section \StRef{Lie3G} where we show that $\real^n$ with 
the solution to Lie's third theorem for vector fields from Section \StRef{SLieT} can be given a global Lie group structure such that the 1-forms $\{\tau^i\}_{1\leq i \leq n}$ in Theorem  \StRef{PLT} are a basis for the left invariant $1$-forms. This provides an alternate proof to the group version of Lie's third theorem for solvable algebras, and an alternate proof that a simply connected solvable Lie group is topologically $\real^n$.

The second step is given in  Section \StRef{movf} where we return to the problem of finding a map $\rho:M \to G$  satisfying \EqRef{MF1} when $G$ is solvable and $M$ is simply connected. Again this can be done using only quadrature and the matrix exponential as shown in Theorem \StRef{CMF}. 

The last step is given in Section \StRef{multmap} where solve the problem of finding the multiplication map $\mu$ for the simply connected solvable Lie group $G$, whose existence in guaranteed from step one. We construct (algebraically)  a set of $n$ $1$-forms $\omega^i$ on $G\times G$ which satisfy equations \EqRef{FSE} with the property that the map $\rho:G\times G \to G$ from Theorem \StRef{CMF} in Section \StRef{movf}  is the multiplication map (Theorem \StRef{L3M}). This gives an algorithm that produces the global multiplication map $\mu:G \times G \to G$ for any $n$-dimensional simply connected solvable Lie group $G$ using $n$ quadratures and the product of matrix exponentials.

The calculations for this paper were performed using the Maple DifferentialGeometry package. 
The current release of the DifferentialGeometry package, the worksheets which
implement Theorems \StRef{PLT}, \StRef{CMF}, \StRef{L3M}, and the examples can be downloaded from ${\bf {\rm http\!:\!/\!/digitalcommons.usu.edu/dg/}}\, $.

\section{The Reduction Theorem}\StTag{Sreduce}

Let $\lieg$ be an $n$-dimensional Lie algebra and suppose that $\liek\subset \lieg$ is a codimension one ideal. With this hypothesis on $\lieg$, Theorem \StRef{reduce} below uses one quadrature and the matrix exponential to modify a set of differential $1$-forms satisfying equation \EqRef{FSE} so that the resulting forms satisfy equation \EqRef{FSE} with the structure constants for the Lie algebra $\liek\times \real$.



Let $\liek \subset \lieg$ denote a codimension 1 ideal in the $n$-dimensional Lie algebra $\lieg$. A basis $\beta=\{ \bfe_1,\ldots , \bfe_{n-1},\bfe_n\}$  for $\lieg$ such that $\hat \beta = \{ \bfe_a \}_{1\leq  a\leq n-1}$ is a basis for $\liek$ is {\bf adapted to the codimension one ideal $\liek\subset \lieg$}. In this case the structure constants are
\begin{equation}
[\bfe_a, \bfe_b ]= C_{ab}^c\bfe_c, \quad [\bfe_n, \bfe_a] = C_{na}^b \bfe_b\quad ,   1\leq a,b,c \leq n.
\EqTag{SEab}
\end{equation}
Note that $C^a_{bc}$ in equation \EqRef{SEab} are the structure constants of the $(n-1)$-dimensional Lie algebra $\liek$ in the basis $\hat \beta$.

This leads to a well-known lemma,  see  \cite{lychagin:1991a}.

\begin{Lemma}  \StTag{WK} Let $\liek \subset \lieg$ be a codimension one ideal and let $\beta=\{ \bfe_a,\bfe_n\}_{1\leq a \leq n}$
be an adapted basis.  Let  $\tilde \omega^i$ be $n$ differential forms  on a manifold $M$ satisfying equation \EqRef{FSE} for the Lie algebra $\lieg$ where the structure constants $\tilde C^i_{jk}$ are given in a basis $\tilde \beta=\{ {\bf f}_i \}_{1\leq i \leq n}$, and let $[P^i_j]$ be the change of basis matrix
$$
{\bf f}_j=P^i_j {\bf e}_i.
$$
Then $ \omega^i=P^i_j \tilde \omega^j$ satisfy the structure equations 
\begin{equation}
\begin{aligned}
 d\omega^a +\frac{1}{2} C^a_{bc} \omega^b \wedge \omega^c + C^a_{nb} \omega^n \wedge \omega^b & = 0\\
 d \omega^n &  = 0  
\end{aligned}
\EqTag{ASE}
\end{equation}
where $C^a_{bc}$ and $C^n_{nb}$ are the structure constants of the Lie algebra $\lieg$ in the basis $\beta$ (equation \EqRef{SEab}).
\end{Lemma}

\medskip
Given differential $1$-forms $\omega^i$ satisfying the structure equation \EqRef{FSE} where $\lieg$ admits a codimension $1$ ideal,  Lemma \StRef{WK} shows that by a constant linear change in the forms $\omega^i$, the resulting forms can be made to satisfy equations \EqRef{ASE}. Equation \EqRef{ASE} is then the starting point to the following reduction theorem which is fundamental in this article.

\begin{Theorem} \StTag{reduce} Suppose $\lieg$ is an $n$-dimensional Lie algebra with $\liek \subset \lieg$ a codimension one ideal and
let $\beta=\{\bfe_a, \bfe_n\}_{1\leq a \leq n-1}$ be a basis adapted to the codimension one ideal.   Let  $\bfomega=[ \omega^1,\ldots, \omega^n]^T$ be  an $n$-vector of differential $1$-forms  on a simply connected manifold $M$ satisfying \EqRef{ASE} where $C^i_{jk}$ are the structure constants of $\lieg$ in the basis $\beta$. 

Let $f:M \to \real$ so that $\omega^n $ in \EqRef{ASE} satisfies 
\begin{equation}
\omega^n= df,
\EqTag{int}
\end{equation}
and  let $\bfhomega$ be the $n$-vector of $1$-forms on $M$ defined by
\begin{equation}
\bfhomega = \rme^{f [\ad ({\bf e}_n)]}\cdot \bfomega,
\EqTag{homega}
\end{equation}
where $[\ad ({\bf e}_n)] $ is the $n$ by $n$ matrix representation of $\ad ({\bf e}_n)$ in the basis $\beta$. Then the structure equations for $\bfhomega$ are
\begin{equation}
d \homega^a + \frac{1}{2} C^a_{bc} \homega^b \wedge \homega^c =0\, , \quad d \homega^n  = 0
\EqTag{FSE3}
\end{equation}
where $C^a_{bc}$ are the structure constants of the Lie algebra $\liek$ in the basis $\hat \beta = \{ {\bf e}_a\}_{1\leq a\leq n-1}$.
\end{Theorem}




\begin{proof}   The matrix representation of the derivation $\ad({\bf e}_n):\lieg \to \lieg$  is computed in the adapted basis from equation \EqRef{SEab} to be
\begin{equation}
\ad( {\bf e}_n) (\be_b)= [\be_n, \be_b] =C_{nb}^a \be_a, \quad \ad( {\bf e}_n) (\be_n) =[\be_n,\be_n]= 0.
\EqTag{aden}
\end{equation}
Let $\widetilde \ad(\be_n) : \liek \to \liek $ be the restriction of $\ad({\bf e}_n)$ to the invariant subspace $\liek$. The $n-1\times n-1$ matrix representation $[\widetilde \ad(\be_n)]$ of $\widetilde \ad(\be_n) $ in the basis $\beta$ is computed from
 \EqRef{aden} to be
\begin{equation}
[\widetilde \ad(\be_n)]^a_b= C^a_{nb}.
\EqTag{alpc}
\end{equation}
The $n\times n$ matrix valued function $ \rme^{ f\, [\,\ad({\bf e}_n)\, ]} $ in $M$ in equation \EqRef{homega} can then be written,
\begin{equation}
\rme^{ f\,[\, \ad({\bf e}_n)\, ]}  =  \left[ \begin{array}{cc} \rme^{ f\, [\, \widetilde \ad(\be_n)\, ]  }  & {\bf 0} \\ {\bf 0}^T & 1 \end{array} \right]
\EqTag{Ainbas}
\end{equation}
where ${\bf 0}$ is the $n-1$ zero vector and ${\bf 0}^T$ is its transpose. The forms $\bfhomega$ in equation \EqRef{homega} can then be written using equation \EqRef{Ainbas} (and matrix vector multiplication) as
\begin{equation}
[\homega^a] = \rme^{ f \,[\, \widetilde \ad(\be_n)\,] } [\omega^b]  , \quad \homega^n = \omega^n .
\EqTag{homega2}
\end{equation}

The verification of equation \EqRef{FSE3} in Theorem \StRef{reduce} can now be checked directly by taking the exterior derivative of the forms $\homega^a$ in equation \EqRef{homega2}. In order to do so we need the exterior derivative of $\rme^{ f\, [\, \widetilde \ad(\be_n)\, ]}$  from equation \EqRef{alpc} is
\begin{equation}
\begin{aligned}
d \, ( \ \rme^{ f\, [\, \widetilde \ad(\be_n)\, ]}\, )   = df \,  \rme^{ f\, [\, \widetilde \ad(\be_n)\, ]} [\, \widetilde \ad(\be_n)\, ].
\EqTag{dA} 
\end{aligned}
\end{equation}
We now compute $d \homega^a$ using equations \EqRef{dA}, \EqRef{ASE}, \EqRef{alpc} and \EqRef{int},
\begin{equation}
\begin{aligned}
& &[ d  \homega^a ]  &=   \rme^{ f\, [\, \widetilde \ad(\be_n)\, ]} [\, \widetilde \ad(\be_n)\, ][-\omega^b\wedge df]+ \rme^{ f\, [\, \widetilde \ad(\be_n)\, ]}[ d \omega^b]\\
& & & =   \rme^{ f\, [\, \widetilde \ad(\be_n)\, ]}[ - C^c_{nb} \omega^b \wedge df]  - \rme^{ f\, [\, \widetilde \ad(\be_n)\, ]}[ \frac{1}{2} C^b_{cd} \omega^c \wedge \omega^d + C^b_{cn}\omega ^c \wedge  \omega^n]\\
& & &=   \rme^{ f\, [\, \widetilde \ad(\be_n)\, ]}[- C^c_{nb}  \omega^b\wedge \omega^n] -  \rme^{ f\, [\, \widetilde \ad(\be_n)\, ]} 
[ \frac{1}{2} C^b_{cd} \omega^c \wedge \omega^d + C^b_{cn}\omega ^c \wedge  \omega^n] \\
& & & =  \rme^{ f\, [\, \widetilde \ad(\be_n)\, ]}[- 
 \frac{1}{2} C^b_{cd} \omega^c \wedge \omega^d ]\\
& && =[-\frac{1}{2} C^a_{cd}  \homega^c \wedge \homega^d],
\end{aligned}
\EqTag{bigd}
\end{equation}
where the last line follows because $\rme^{ f\, [\, \widetilde \ad(\be_n)\, ]} $ is the matrix representation of an automorphism of $\liek$ for each $x\in M$. 
\end{proof}

The computation leading to Theorem \StRef{reduce} uses one quadrature in equation \EqRef{int} and the matrix exponential \EqRef{homega}. We also note that the change of forms in equation \EqRef{homega} results in the elimination of the third term  $C^a_{nb}$ in the first structure equation in \EqRef{ASE}.  Since the constants  $C^a_{bc}$ in equation \EqRef{FSE3} are the structure constants for $\liek$ in the basis $\{ \bfe_a \}_{1\leq  a\leq n-1}$, equation \EqRef{homega} can then be thought of reducing the Lie algebra $\lieg$ to $\liek\times \real$.  See also Remark \StRef{BFR} below.

We now apply Theorem \StRef{reduce} to the case where $\lieg$ admits a sequence of subalgebras $\liek_s\subset \lieg$, $s=0,\ldots,r$, satisfying
\begin{equation}
\liek_r \subset \liek_{r-1} \subset \ldots
\liek_1 \subset \liek_0= \lieg 
\EqTag{sid}
\end{equation}
where each $\liek_s \subset \liek_{s-1}$, $1\leq s \leq r$ is a codimension 1 ideal. We call the sequence in equation \EqRef{sid} {\bf a sequence of codimension one ideals}. Note that $\dim \liek_s =\dim \lieg -s =n-s$. A basis $\beta= \{ \bfe_i \}_{1\leq i \leq n} $ for $\lieg$ is said to be {\bf adapted to a sequence of codimension one ideals} $\liek_s \subset \liek_{s-1} , s=1,\ldots,r$, if 
\begin{equation}
\liek_s =\spn \{\bfe_1,\ldots, \bfe_{n-s} \}, \quad s=0\ldots r.
\EqTag{Abas}
\end{equation}
Given the basis $\beta= \{ \bfe_i \}_{1\leq i \leq n} $ adapted to the sequence of codimension one ideals $\liek_s\subset \liek_{s-1},s=1,\ldots,n$, we  let $\ad_{s}( \bfe_{n-s}) \in {\rm Der}(\liek_s)$ be the restriction of $\ad(\bfe_{n-s})$ to the invariant subspace $\liek_s$,
\begin{equation}
\ad_{s}( \bfe_{n-s}) = \ad( \bfe_{n-s}) |_{\liek_s}, \quad 0\leq s\leq r-1.
\EqTag{alseq}
\end{equation}
In particular $\ad_0(\bfe_n) = \ad(\bfe_n)$.

Theorem \StRef{reduce} easily extends to the case of a sequence of codimension one ideals giving the following.

\begin{Corollary} \StTag{C1} Suppose the $n$-dimensional Lie algebra $\lieg$ in Theorem \StRef{reduce} admits a sequence of codimension one ideals $\liek_s \subset \liek_{s-1},\ s=1,\ldots,r $, and let $\beta= \{ \bfe_i \}_{1\leq i \leq n}$ be a basis adapted to the sequence. Let $\omega^i$ be $n$ differential $1$-forms on a simply connected manifold $M$ satisfying \EqRef{FSE} where the structure constants are computed in the  basis $\beta$.
Then there exists $r$ functions $ f^t:M \to \real, \ t= n-r+1,\ldots, \leq n $ and $n-r$ differential $1$-forms $\{ \comega^\alpha \}_{1\leq \alpha \leq n-r}$ obtained inductively from $\omega^i$ by $r$ quadratures and the exponential of matrices such that

\begin{equation}
\left[ \begin{array}{cc}  \rme^{f^{n-r+1}[\ad_{r-1}( \bfe_{n-r+1})]} & {\bf 0}_{r-1} \\ {\bf 0}_{r-1}^T & {\rm I}_{r-1} \end{array} \right]\cdots\left[ \begin{array}{cc}  \rme^{f^{n-1}[\ad_{1}( \bfe_{n-1})]} & {\bf 0}_1 \\ {\bf 0}_1 ^T & 1 \end{array} \right] \, \rme^ {f^n[\ad(\be_n)]} \, 
 \left[ \begin{array}{c} \omega^1\\ \vdots \\ \omega^n  \end{array} \right]
= \left[ \begin{array}{c} \comega^1\\ \vdots \\ \comega^{n-r} \\ d  f ^{n-r+1} \\ \vdots \\ df^n  \end{array} \right]
\EqTag{seqreduce}
\end{equation}
where $[\ad_{s}( \bfe_{n-s})]$ is the $n-s$ by $n-s$ matrix representation of $\ad_{s}( \bfe_{n-s}) \in {\der}(\liek_s)$ defined in \EqRef{alseq} in the basis \EqRef{Abas} for $\liek_{s}$, ${\rm I}_s $ is the $s$ by $s$ identity matrix, and ${\bf 0}_s$ is the $n-s$ by $s$ zero matrix.

Furthermore the $n-r$ differential $1$-forms $\comega^\alpha$ satisfy
\begin{equation}
d \comega^\alpha +\frac{1}{2}C^\alpha_{\beta \gamma} \comega^\beta \wedge \comega^\gamma = 0 
\EqTag{redste}
\end{equation}
where $C^\alpha_{\beta \gamma}$ are the structure constants for the $(n-r)$-dimensional Lie algebra $\liek_r$ in the basis $\check \beta=\{ \bfe_\alpha\}_{1\leq \alpha \leq n-r}$.
\end{Corollary}

\begin{proof} Using Theorem \StRef{reduce} we inductively compute the functions $f^t$, $n-r+1\leq t \leq n$. First we let $f^n:M \to \real$ be the function obtained in Theorem \StRef{reduce} using $\liek=\liek_1$ so that $ \omega^n=df^n$. The first term in equation \EqRef{seqreduce} is
\begin{equation}
 \rme^{ f^n [\, \ad({\bf e}_n)\,]  } \left[ \begin{array}{c} \omega^1\\ \vdots \\ \omega^n \end{array} \right].
 = \left[ \begin{array}{c} \homega^1\\ \vdots \\ \homega^{n-1} \\ df^n  \end{array} \right].
\EqTag{CBn0}
\end{equation}

Now because $\liek_{2} \subset \liek_{1}$ is a codimension one ideal, and the basis $\{\bfe_i\}_{1\leq i\leq n}$ is adapted to the sequence \EqRef{sid}, we have
$$
d \homega^{n-1} = 0.
$$
Again letting $f^{n-1} :M \to \real$ so that $\homega^{n-1} = d f^{n-1}$, and applying Theorem \StRef{reduce} to the case of $\liek_2 \subset \liek_1$ we get 
\begin{equation}
\rme^{ f^{n-1} [\, \ad_1(\be_{n-1})\, ] }  
\left[ \begin{array}{c} \homega^1 \\  \vdots \\ \homega^{n-1} \end{array} \right]=
 \left[ \begin{array}{c} \tomega^1 \\ \vdots \\ \tomega^{n-2} \\ df^{n-1}  \end{array} \right]
\EqTag{CBn1}
\end{equation}
where the $n-2$ differential forms $\tomega^u$ satisfy 
$$
d \tilde \omega^u +\frac{1}{2}   C^u_{vw} \tilde \omega^v \wedge \tilde \omega^w = 0
$$
where $C^u_{vw}$ are the structure constants for $\liek_2$ in the basis $\tilde \beta=\{\be_u\}_{1\leq u \leq n-2}$. 
Using equation \EqRef{CBn1} and equation \EqRef{CBn0} produces the right most two matrices on the left hand side of equation \EqRef{seqreduce}.

Continuing by induction this produces, by $n-r$ quadratures, the $n-r$ functions $f^t:M \to \real$ and the matrices in equation \EqRef{seqreduce}. Furthermore the forms $\bfcomega$ in equation \EqRef{seqreduce} satisfy the structure equation \EqRef{redste}. \end{proof}

If $\lieg$ is an $n$-dimensional solvable Lie algebra it is well known that $\lieg$ admits a sequence of codimension $1$ ideals $\liek_s\subset \liek_{s-1}, \ s= 1,\ldots,n$, \cite{jacobson:1962a}, and these can be found using elementary linear algebra. A basis adapted to the derived series can easily be used to produce a basis adapted to a sequence of codimension $1$ ideals. Note that $\dim \liek_n =0$.  This leads to the next corollary.

\begin{Corollary} \StTag{C2} Let $\lieg$ be an $n$-dimensional solvable  Lie algebra, and let $\beta=\{ \bfe_i\}_{1\leq i \leq n}$ be a basis adapted to a sequence of codimension one ideals $\liek_s\subset \liek_{s-1}$, $s=1,\ldots,n$. Let $\omega^i$ be $n$ differential $1$-forms on a simply connected manifold $M$ satisfying equation \EqRef{FSE} where the structure constants are computed in the basis $\beta$. 

There exists $n$ functions $f^i:M \to \real$, $i=1,\ldots,n$, obtained inductively from the forms $\omega^i$ using $n$ quadratures and the exponential of matrices so that
\begin{equation}
\left[ \begin{array}{cc}  \rme^{f^{2}[\, \ad_{n-2}(\be_2)\,]} & {\bf 0}_{n-2} \\ {\bf 0}^T_{n-2} & {\rm I}_{n-2} \end{array} \right] \cdots 
\left[ \begin{array}{cc}  \rme^{f^{n-1}[\ad_{1}( \bfe_{n-1})]} & {\bf 0}_{1} \\ {\bf 0}_{1} ^T & 1 \end{array} \right] \, \rme^ {f^n[\ad(\be_n)]} \, 
 \left[ \begin{array}{c} \omega^1\\ \vdots \\ \omega^n  \end{array} \right]
  = \left[ \begin{array}{c}  d  f^1\\ \vdots \\df^n \end{array} \right],
\EqTag{FIE}
\end{equation}
where $[\ad_{s}( \bfe_{n-s})]$ is the $n-s$ by $n-s$ matrix representation of $\ad_{s}( \bfe_{n-s}) \in {\der}(\liek_s)$ defined in \EqRef{alseq} in the basis \EqRef{Abas} for $\liek_{s}$ where $ s=0,\ldots,n-2$, ${\rm I}_s $ is the $s$ by $s$ identity matrix, and ${\bf 0}_s$ is the $n-s$ by $s$ zero matrix.
\end{Corollary}

\begin{Remark} The function $f:M \to \real$ in Theorem \StRef{reduce} is uniquely determined if the value of $f$ is prescribed at a point. Similarly for the functions in Corollaries \StRef{C1} and \StRef{C2}. 
\end{Remark}

\begin{Remark} \StTag{BFR} An alternative description of Theorem \StRef{reduce} is the following. Given differential $1$-forms $\{ \omega^i \}_{ 1\leq i \leq n}$ on a manifold $M$ satisfying \EqRef{FSE} and a basis $\beta=\{ \be_i\}_{1\leq i \leq n}$ for the Lie algebra $\lieg$,  define the $\lieg$ valued differential $1$-form $\bfomega_{\lieg}$ by
$$
\bfomega_\lieg = \omega^i \otimes \be_i .
$$
Equation \EqRef{FSE} is then written
\begin{equation}
d \bfomega_\lieg + \frac{1}{2} \bfomega_\lieg \wedge \bfomega_\lieg =0
\EqTag{FSE2}
\end{equation}
where the wedge product $ \bfomega_\lieg \wedge_\lieg \bfomega_\lieg $ is defined by
\begin{equation}
\bfomega_\lieg\wedge \bfomega_\lieg (X,Y) = [\bfomega_\lieg(X), \bfomega_\lieg(Y)]_{\lieg} \qquad {\rm for \ all \ }  X,Y \in TM.
\EqTag{LBE}
\end{equation}

Suppose now that  $\liek\subset \lieg$ is a codimension one ideal and that $\bfomega_\lieg$ is a Lie algebra valued form on a simply connected manifold $M$ satisfying the structure equations $d\bfomega_\lieg+\frac{1}{2} \bfomega_\lieg\wedge \bfomega _\lieg=0$.
Let $A :M \to {\rm Aut}(\lieg)$ be $ A(x) = \exp\left(f(x) \ad({\bf e}_n)\right)$ where $f:M \to \real$ is
defined in \EqRef{int}.  Theorem \StRef{reduce} shows 
 that  $\hbfomega_{\liek \times \real }= A\circ \omega $ takes values in the Lie algebra $\liek \times \real$ and satisfies
$$
d \hbfomega_{\liek \times \real} + \frac{1}{2} \hbfomega_{\liek \times \real} \wedge \hbfomega_{\liek \times \real} =0.
$$
\end{Remark}

\section{First integrals for completely integrable Pfaffian systems with a solvable symmetry algebra} \StTag{Sfi}

Let $I \subset T^*M$ be a rank $n$ Pfaffian system. A function $f:M\to \real$ such that $df(x)\in I_x$ for all $x\in M$ is called a {\bf first integral} of $I$, see \cite{lychagin:1991a} and \cite{doubrov:2000a}. If $f_i:M \to \real$ are $n$ first integrals satisfying $I=\spn{ df_i}$ then the function $f_i$ are a complete set of first integrals.

Suppose  $I= \spn \{\omega^i\}_{1\leq i \leq n}$ where  $\omega^i$ are differential $1$-forms which satisfy the structure equations \EqRef{FSE}. If $\lieg$ satisfies the conditions in Theorem \StRef{reduce} or its corollaries, we show how first integrals of $I$ can be computed using only quadrature and matrix exponentiation.  In particular Theorem \StRef{reduce} states that if $\liek \subset \lieg$ is a co-dimension ideal, then one quadrature produces the first integral $f$ in equation \EqRef{int}.  Corollary \StRef{C2} implies the following.

\begin{Corollary} \StTag{S2C1} Let $I\subset T^*M$ be a completely integrable Pfaffian system on a simply connected manifold $M$. Suppose  $I= \spn \{\omega^i\}_{1\leq i \leq n}$ where the differential $1$-forms $\omega^i$ satisfy
$$
d\omega^i + \frac{1}{2} C^i_{jk} \omega^j \wedge \omega^k = 0,
$$
and $C^i_{jk}$ are the structure constants of a solvable Lie algebra $\lieg$. The $n$ functions $f^i : M \to \real$ constructed sequentially by quadratures and matrix exponentials through formula \EqRef{FIE} are a complete set of first integrals for $I$.
\end{Corollary}

We now show how Corollary \StRef{S2C1} is used in practice, see also \cite{lychagin:1991a}. Let $\Gamma=\spn \{X_i\}_{1\leq i \leq n} $ be an $n$-dimensional Lie algebra of vector fields on $M$ which are infinitesmimal symmetries of a completely integrable rank $n$ Pfaffian system $I=\spn \{ \theta^j\}_{1\leq j \leq n} $. Assume
 $\bfGamma_p=\spn\{X_i(p)\}$ is a $n$-dimensional subspace of $T_pM$ for each $p\in M$ which satisfies the transversality condition 
\begin{equation}
 \bfGamma_p \cap \ann(I_p) = 0, \qquad {\rm for \ all } \ p \in M.
\EqTag{tc}
\end{equation}
The transversality condition \EqRef{tc} implies that $n\times n$ matrix $(P^i_j)= (\theta^i(X_j))$ is invertible, and a simple
computation (see \cite{lychagin:1991a} or \cite{fels:2007a}) shows that the differential forms
\begin{equation}
\omega^i = (P^{-1})^i_j \theta^j
\EqTag{domega}
\end{equation} 
satisfy $I = \spn\{ \omega^i \}_{1\leq i \leq n}$ and
$$
d \omega^i +\frac{1}{2} C^i_{jk} \omega^j \wedge \omega^k  = 0
$$
where $[X_j,X_k]= C^i_{jk} X_i $.  If the Lie algebra $\Gamma$ is solvable then Corollary \StRef{C2} may be used to produce $n$ first integrals for $I$ using quadratures and matrix exponentiation. This is similar to Theorem 3 on page 39 in \cite{lychagin:1991a} but with the important distinction that Theorem 3 of \cite{lychagin:1991a} requires solving functional equations while Corollary \StRef{C2} does not.

\begin{Example}\StTag{Ex1}
Consider the ODE  from pg. 152 in \cite{olver:1998a},
\begin{equation}
u_{xxx}=\frac{3u_{xx}^2}{u_x}+\frac{u_{xx}^3}{u_x^5} . 
\EqTag{621}
\end{equation} 
On the  manifold $ {M} = \{ (x,u,u_x,u_{xx}) \in \real^4 \ | \ u_x \neq 0 \
\} $, equation \EqRef{621} defines the Pfaffian system $I=\spn\{ \theta^1,\theta^2,\theta^3\}$ where
\begin{equation}
 \theta^1= du -u_x dx,\quad  \theta^2=du_x -u_{xx}dx,\quad \theta^3=
 du_{xxx} - \left( 3u_{xx}^2u_x^{-1}+u_{xx}^3u_x^{-5}\right) dx.
 \EqTag{PS3}
\end{equation}
The Pfaffian system \EqRef{PS3} admits the infinitesimal symmetries,
\begin{equation}
Z_1  = \partial_x, \ Z_2 = \partial_u ,\ Z_3 = u\partial_x
-u_x^2\partial_{u_x} -3u_xu_{xx}\partial_{u_{xx}} \ 
\EqTag{SVF}
\end{equation}
and the only non-vanishing bracket is
\begin{equation}
[Z_2, Z_3]= Z_1.
\EqTag{MCE}
\end{equation}
Let $\{ \, \be_1,\, \be_2,\, \be_3\, \}$ be a basis for the $3$-dimensional Lie algebra isomorphic to the
Lie algebra of vector fields $\{ Z_1,Z_2,Z_3 \}$ so that $[\be_2,\be_3]=\be_1$.  We have the sequence of codimension one ideals $\spn\{\be_1\} \subset \spn\{\be_1, \be_2\} \subset \lieg$.

On the set ${ M}^0=(x,u,u_x\neq 0,u_{xx}\neq 0)$
the forms  $ \{\omega^i\}_{1\leq i\leq 3} $ on ${ M}^0$ are computed from equation \EqRef{domega} using equations \EqRef{PS3} and  \EqRef{SVF} to be
\begin{equation}
\begin{split} \omega^1 &=
dx+\frac{1}{(u_xu_{xx})^2}(3u_x^6+3u_{xx}uu_x^4+u_{xx}^2u)du_x-\frac{u_x^3}{u_{xx}^3}(u_x^2+u_{xx}u)du_{xx}
, \\ \omega^2 & = du+\frac{3
u_x^5}{u_{xx}^2}du_x-\frac{u_x^6}{u_{xx}^3}du_{xx}  ,\quad  \omega^3 =
-\frac{(3u_x^4+u_{xx})}{u_{xx}u_x^2}du_x+\frac{u_x^3}{u_{xx}^2}du_{xx}\
. \end{split}
\EqTag{omex1}
\end{equation}

Now $ d\omega^3 =0$ and we may write $\omega^3 = df^3$ where
\begin{equation}
f^3 =   \frac{1}{u_x}-\frac{u_x^3}{u_{xx}} .
\EqTag{om3}
\end{equation}
in equation \EqRef{FIE}. Then by exponentiating
\begin{equation}
[\, \ad(\be_3)\,] = \left[ \begin{array}{ccc} 0 & -1 & 0  \\ 0 & 0 & 0 \\ 0 &  0  & 0  \end{array} \right]
\EqTag{ad3}
\end{equation}
this gives us our first step in the algorithm of Corollary \StRef{C2} as
\begin{equation}
\left[ \begin{array}{ccc} 1 & -f^3 & 0 \\ 0 & 1 & 0 \\ 0 & 0 &  1  \end{array} \right] \left[ \begin{array}{c} \omega^1 \\ \omega^2 \\ \omega^3 \end{array} \right]
= \left[ \begin{array}{c} \hat \omega^1 \\ \hat \omega^2\\ df^3 \end{array} \right]
\end{equation}
where
$$
\begin{aligned}
\left[ \begin{array}{cc} \hat \omega^1 \\ \hat \omega^2 \end{array} \right] & = \left[ \begin{array}{cc} 1 &   \frac{u_x^3}{u_{xx}}- \frac{1}{u_x} \\ 0 & 1 \end{array} \right]
\left[ \begin{array}{cc}  \omega^1 \\  \omega^2 \end{array} \right] \\
& = \left[ \begin{array}{cc}  dx+\frac{u_x^4-u_{xx}}{u_x u_{xx}} du+\frac{3 u_x^{10}+3 u u_{xx}^2 u_x^4+u u_{xx}^3}{u_{xx}^3 u_x^2} du_x -\frac{u_x^3(u_x^6+u u_{xx}^2)}{u_{xx}^4} d u_{xx}\\  du+\frac{3 u_x^5}{u_{xx}^2} du_x-\frac{u_x^6}{u_{xx}^3} du_{xx} \end{array} \right]. 
\end{aligned}
$$
Now
$$
d \hat \omega^1 =0, \quad d \hat \omega^2 = 0. 
$$
The corresponding first integrals $f^1, f^2: M^0 \to \real$ satisfying the equations $ \hat \omega^1 =df^1, \hat \omega^2 = df^2$ are given by
\begin{equation}
f^1 = x+ \frac{u_x^{10}+3u u_{xx}^2u_x^4-3 u u_{xx}^3}{ 3 u_ x u_{xx}^3}\ , \quad
f^2 = u+\frac{u_x^6}{2 u_{xx}^2}.
\EqTag{ex1f1f2}
\end{equation}
\end{Example}

\medskip

\section{Lie's Third Theorem: The Vector Field Case}\StTag{SLieT}

The vector field version of Lie's third theorem is the following.

\begin{Theorem} (Lie's third theorem for vector fields) Let $\lieg$ be an $n$-dimensional real Lie algebra. There exists $n$ pointwise linearly independent vector fields on $\real^n$ whose span form a Lie algebra of vector fields that is isomorphic to $\lieg$.
\end{Theorem}

Starting with a solvable Lie algebra $\lieg$ we give an algebraic proof of this version of Lie's third theorem by constructing the vector fields using only the matrix exponential. The proof follows almost directly from Corollary \StRef{C2}. First we give the dual version. 

\begin{Theorem} \StTag{PLT} Let $\lieg$ be an $n$-dimensional solvable Lie algebra and  let  $\beta= \{ \bfe_i \}_{1\leq i \leq n} $ be a basis for $\lieg$  adapted to a sequence of codimension one ideals $\liek_s\subset \liek_{s-1}$, $s=1,\ldots,n$ in $\lieg$. Let $\ad_s(\be_{n-s}) \in {\der}(\liek_s), s=0,\ldots, n-2$ be the derivation of $\liek_s$ defined in \EqRef{alseq} by
\begin{equation}
\ad_s(\be_{n-s}) =  \ad ( \bfe_{n-s})|_{\liek_{s}} , \quad s=0,\ldots, n-2.
\EqTag{alseq2}
\end{equation}
Let $\tau^i$ be the  $n$ point-wise linear independent $1$-forms  on $\real^n$ defined by
\begin{equation}
\left[ \begin{array}{c} \tau^1 \\ \vdots \\ \tau^n \end{array} \right] = 
 \rme^{-x^{n}[\ad(\be_n) ]} 
   \left[ \begin{array}{cc} \rme^{-x^{n-1}[\ad_{1}{(\be_{n-1})}]} & {\bf 0}_1 \\ {\bf 0}_1^T & 1 \end{array} \right]
 \ldots
\left[ \begin{array}{cc} \rme^{-x_{2}[\ad_{n-2}(\be_2)]} & {\bf 0}_{n-2} \\ {\bf  0}_{n-2}^T & {\rm I}_{n-2} \end{array} \right]
\left[ \begin{array}{c} dx^1 \\ \vdots \\ dx^n \end{array} \right]
\EqTag{Lieom}
\end{equation}
where $[\ad_s(\be_{n-s})]$ is the $n-s$ by $n-s$ matrix representation of $\ad_s(\be_{n-s})$ in the basis  \EqRef{Abas} for $\liek_{s}$, $s=0,\ldots,n-2$,  ${\rm I}_{s}$ is the $s$ by $s$ identity matrix, and ${\bf 0}_s$ is the $n-s$ by $s$ zero matrix.
Then $\tau^i$ satisfy
\begin{equation}
d \tau^i +\frac{1}{2}C^i_{jk} \tau^j \wedge \tau^k = 0,
\EqTag{LTF}
\end{equation}
where $C^i_{jk}$ are the structure constants for $\lieg$ in the basis $\beta$.
\end{Theorem}

Before presenting the proof of Theorem \StRef{PLT}, let $X_i, 1\leq i \leq n$ be the dual frame to $\tau^i$ in equation \EqRef{Lieom}. This produces the $n$ linearly independent vector fields on $\real^n$, given by
\begin{equation}
\left[ \begin{array}{c} X_1 \\ \vdots \\ X_n \end{array} \right] = 
{\rm e}^{x_n[\ad(\be_n)]^T}\left[ \begin{array}{cc} \rme^{x^{n-1}[\ad_{1}{(\be_{n-1})}]^T} & {\bf 0}_1 \\ {\bf 0}_1^T & 1 \end{array} \right] \cdots \left[ \begin{array}{cc} {\rm e}^{x_{2}[\ad_{n-2}(\be _2)]^T } & {\bf 0}_{n-2} \\ {\bf 0}^T_{n-2} & {\rm I}_{n-2}   \end{array}\right]    \left[ \begin{array}{c} \partial_{x^1} \\ \vdots \\ \partial_{x^n} \end{array} \right],
\EqTag{LieX}
\end{equation}
where $T$ is the transpose. We then have the following, which is Lie's third theorem for vector fields.

\begin{Corollary}\StTag{CLV}  The  $n$ vector fields $X_i$ on $\real^n$ defined in \EqRef{LieX} are point-wise linearly independent and satisfy
\begin{equation}
[X_i,X_j]=C^k_{ij} X_k
\EqTag{BLA}
\end{equation}
where $C^i_{jk}$ are the structure constants for $\lieg$ in the basis $\beta$ of Theorem \StRef{PLT}.
Therefore the vector fields in \EqRef{LieX} produce a solution to Lie's third Theorem (for vector fields) using only the matrix exponential. The linear map $\phi:\lieg \to \spn\{ X_i\}$ induced by
\begin{equation}
\phi({\bf e}_i) = X_i
\EqTag{giso}
\end{equation}
is a Lie algebra isomorphism.
\end{Corollary}

\begin{proof}(Theorem \StRef{PLT})  The idea of the proof of Theorem \StRef{PLT}  is to think of solving equation \EqRef{FIE} of Corollary \StRef{C2} for $\omega^i$. This gives equation \EqRef{Lieom} if we substitute  $f^i=x^i$ and $\omega^i=\tau^i$.  

The proof goes by induction. Consider the differential forms on $\real^n$ given by
\begin{equation}
\left[ \begin{array}{c} \htau^1 \\ \vdots \\ \htau^n \end{array} \right] = 
 \left[ \begin{array}{cc} \rme^{-x^{n-1}[\ad_{1}{(\be_{n-1})}]} & {\bf 0}_1 \\ {\bf 0}_1^T & 1 \end{array} \right]
 \ldots
\left[ \begin{array}{cc} \rme^{-x_{2}[\ad_{n-2}(\be_2)]} & {\bf 0}_{n-2} \\ {\bf  0}_{n-2}^T & {\rm I}_{n-2} \end{array} \right]
\left[ \begin{array}{c} dx^1 \\ \vdots \\ dx^n \end{array} \right] .
\EqTag{somega}
\end{equation}
By the induction hypothesis, the forms $\{ \htau^a\}_{1\leq a \leq n-1}$ from equation \EqRef{somega} satisfy,
\begin{equation*}
d  \htau^a+ \frac{1}{2} C^{a}_{bc}  \htau^b \wedge  \htau^c=0.
\end{equation*}
where $C^a_{bc}$ are the structure constants of the ideal $\liek_{n-1}$ in the basis $\hat \beta=\{ \be_{a}\}_{1\leq a \leq n-1}$. 

The fact that  the differential forms in equation \EqRef{Lieom} satisfy equation \EqRef{LTF} follows simply by reversing the computation in proof of Theorem \StRef{reduce} 
with $\bftau = \exp( -x^n [\ad(\be_n)]) \hat \bftau$ (in particular using equations \EqRef{Ainbas} through \EqRef{bigd}). This shows that the forms $\tau^i$ in \EqRef{Lieom} satisfy equation \EqRef{LTF}.  \end{proof}

\begin{Example}\StTag{Ex2} Consider the $5$-dimensional solvable Lie algebra $A_{5,ab}$ from \cite{winternitz:1976a}, with multiplication table
\begin{equation}
[\be_1, \be_4] = b\, \be_1,\ [\be_1, \be_5] = a \, \be_1,\ [\be_2, \be_4] = \be_2,\ [\be_2, \be_5] = -\be_3,\ [\be_3, \be_4] = \be_3,\ [\be_3, \be_5] = \be_2,
\EqTag{SCe1}
\end{equation}
where the two real parameters $a,b$ satisfy $a^2+b^2\neq 0$. The basis $\{ \be_i\}_{1\leq i \leq 5}$ is adapted to the sequence of co-dimension one ideals $\liek_s=\spn\{ \be_1,\ldots, \be_{5-s}\}$, $s=0,\ldots,5$.

We find the non-zero $\ad_s(\be_{n-s})$ matrices defined in equation \EqRef{alseq2} using equations \EqRef{SCe1} to be
\begin{equation}
[\ad(\be_5 )]= \left[ \begin{array}{ccccc} -a & 0 &  0 & 0 &0 \\ 0 & 0 &-1  & 0 & 0 \\ 0 & 1 &0 & 0 & 0 \\  0 & 0  &0 & 0 & 0 \\ 0 & 0  &0 & 0 & 0 \end
{array} \right], \quad 
[\ad_1(\be_4)]= \left[ \begin{array}{cccc} -b & 0 &  0 & 0 \\ 0 & -1 & 0  & 0 \\ 0 & 0 & -1 & 0 \\ 0 & 0& 0 & 0 \end{array} \right].
\EqTag{Adms}
\end{equation}
The coframe $\bftau$ on $\real^5 $ is computed from \EqRef{Lieom} using equation \EqRef{Adms} and $\ad_3(\be_2)=0$, $\ad_2(\be_3)=0$ is then
\begin{equation}
\begin{aligned}
\left[
\begin{array}{c} 
\tau^1 \\ \tau^2 \\ \tau^3 \\ \tau^4 \\ \tau^5 \end{array}\right]
& =  
\left[ \begin{array}{ccccc} \rme^{a x^5} & 0 &  0 & 0 & 0 \\ 0 & \cos x^5  & \sin x^5  & 0& 0 \\ 0 & -\sin x^5  & \cos x^5 & 0  & 0\\ 0 & 0  &0 & 1 & 0\\ 0 & 0 & 0& 0 & 1\end{array} \right]
 \left[ \begin{array}{ccccc} \rme^{b x^4} & 0 &  0 &0 & 0  \\ 0 & \rme^{x^4} & 0 &0 & 0  \\ 0 & 0 & \rme^{x^4} &0 & 0   \\ 0 & 0& 0 & 1 & 0 \\ 0 & 0 & 0 & 0& 1   \end{array} \right]
\left[
\begin{array}{c} 
dx^1 \\ dx^2 \\ dx^3 \\ dx^4 \\ dx^5 \end{array}\right] \\
&= \left[ 
\begin{array}{c} 
\rme^{a x^5 + bx^4 b} dx^1  \\
\cos x^5\, \rme^{x^4} dx^2+\sin x^5 \, \rme^{x^4} dx^3
\\
 -\sin x^5\, \rme^{x^4} dx^2+\cos x^5 \, \rme^{x^4} dx^3 \\ dx^4 \\ dx^5
\end{array} \right]
\end{aligned}
\EqTag{Ex1forms}
\end{equation}
The structure equations are then easily checked to be,
$$
d\tau^ 1= -b\, \tau^1\wedge \tau^4-a\, \tau^1 \wedge \tau^5,\  d\tau^2 = -\tau^2 \wedge \tau^4-\tau^3 \wedge \tau^5,\ d\tau^3 = \tau^2 \wedge \tau^5-\tau^3 \wedge \tau^4,\ d\tau^4 = 0,\ d\tau^5 = 0.
$$
\end{Example}

\section{Lie's Third Theorem: The Group Case}\StTag{SLie3G}

The group version of Lie's third theorem is the following.

\begin{Theorem} (Lie's third theorem) Let $\lieg$ be an $n$-dimensional real Lie algebra. There exists an $n$-dimensional Lie group $G$ whose Lie algebra of left invariant vector fields is isomorphic to $\lieg$.
\end{Theorem}

Our goal in this section is to give a simple proof to Lie's third theorem for solvable Lie algebras, which will allow us to explicitly construct the group multiplication for a simply connected solvable Lie group using $n$ quadratures and matrix exponentiation (Theorem \StRef{L3M}).  As an intermediate step in the construction of the multiplication map, we show how to construct the map $\rho:M \to G$ in equation \EqRef{MF1} for solvable Lie algebras using quadrature and matrix exponentiation.

\subsection{Lie's third theorem for solvable groups}\StTag{Lie3G}

We begin by showing that given the vector fields $\{X_i\}_{1\leq i \leq n}$ on  $\real^n$ in equation \EqRef{LieX}, we can give a global Lie group structure on $\real^n$ such that $\{X_i\}_{1\leq i \leq n}$ are a basis for the left invariant vector fields. This proves Lie's third theorem (above) for solvable Lie groups while also showing that a simply connected solvable Lie group is (topologically) $\real^n$. The key aspect of the proof, which is used Sections \StRef{movf} and \StRef{multmap}, is that we have a closed form formula for the left invariant vector fields.  For a recent discussion of Lie's third theorem see \cite{vanest:1988a}.

\begin{Theorem} \StTag{LG1} Let $\lieg$ be an $n$-dimensional solvable Lie algebra and let $\beta=\{ {\bf e}_i \}_{1 \leq i \leq n}$ be a basis for $\lieg$ adapted to a sequence of codimension one ideals  $\liek_s\subset \liek_{s-1}$, $s=1,\ldots,n$. Let $\{X_i\}_{1\leq i \leq n}$ be the vector fields on $\real^n$ given in equation  \EqRef{LieX}. Given a point $x_0\in \real^n$, there exists a smooth multiplication map $\mu: \real^n \times \real^n\to \real^n$ so that $\real^n$ is a solvable Lie group with identity $x_0$ and with Lie algebra of left invariant vector fields given by $\lieg=\spn \{ X_i\}_{1\leq i \leq n}$.
\end{Theorem}

The proof of Theorem \StRef{LG1} follows directly from Theorem 8.7 in \cite{sharpe:1997a} (see also \cite{vanest:1988a}) and the next lemma.

\begin{Lemma}\StTag{CV} Every vector field in the span of the vector fields $\{X_i\}_{ 1\leq i \leq n}$ on $\real^n$ in equation \EqRef{LieX} is complete.
\end{Lemma}

\begin{proof} We prove this by induction. Assume that the span of $n-1$ vector fields $\{ \tilde X_r\}_{ 1\leq r \leq n-1}$  on $\real^{n-1}$ with coordinates $x^1,\ldots, x^{n-1}$ given by 
\begin{equation}
\left[ \begin{array}{c} \tilde X_1 \\ \vdots \\ \tilde X_{n-1} \end{array} \right] = \left[\begin{array}{cc} 
\rme^{x^{n-1}[\ad_1(\be_{n-1})]}   & {\bf 0}_1 \\  {\bf 0}_1^T & 1 \end{array} \right] \ldots \left[ \begin{array}{cc} \rme^{x^{2}[\ad_{n-2}(\be_2)]^T } & {\bf 0}_{n-2} \\ {\bf 0}^T_{n-2} & {\rm I}_{n-2}   \end{array}\right]  \left[ \begin{array}{c} \partial_{x^1} \\ \vdots \\ \partial_{x^{n-1}} \end{array} \right],
\EqTag{LieX2}
\end{equation}
consists of complete vector fields. The vector fields $\{ \tilde X_r\}_{1\leq r \leq n-1}$ are globally defined on $\real^{n-1}$ and satisfy equation \EqRef{BLA} with the structure  constants for the co-dimension one sub-algebra $\liek_1 \subset \lieg$ defined in Theorem \StRef{PLT}. By  \cite{sharpe:1997a} (or \cite{vanest:1988a})  given $\tilde x_0 \in \real^{n-1}$ there exists a Lie group structure on $\real^{n-1}$ with $\tilde x_0$ as the identity such that $\{ \tilde X_r\}_{1\leq r \leq n-1}$ are a basis for the left invariant vector fields on $\real^{n-1}$. Denote this Lie group by $H$  and let $\lieh=\spn \{ \tilde X_r\}_{1\leq r \leq n-1}$   be the corresponding Lie algebra.

Now the $n$ vector fields $X_i$ in \EqRef{LieX}  can be written in terms of the vector fields in \EqRef{LieX2} considered as vector fields on $\real^n$ as
\begin{equation}
\left[ \begin{array}{c} X_1 \\ \vdots \\ X_n \end{array} \right] = \rme^{x^n[\ad(\be_n)]^T}  \left[ \begin{array}{c} \tilde X_1 \\ \vdots \\ \tilde X_{n-1} \\ \partial_{x^n} \end{array} \right] .
\EqTag{LieX3}
\end{equation}
Let $\tilde \xi^s_r,\ s,r=1,\ldots,n-1$ be the coefficients of the vector fields $\tilde X_ r$  and let $\eta^j_i, \ i,j=1,\ldots,n$ be the coefficients of the vector fields $X_i$. From equations \EqRef{LieX3} and \EqRef{Ainbas} we have
\begin{equation}
\eta^s_r = \left[ {\rm e}^{x^n[\widetilde \ad(\be_n)]}  \right]^u_r  \tilde \xi^s_u, \quad \eta^s_n =0, \quad \eta^n_r = 0, \quad \eta^n_n= 1. 
\EqTag{LieX4}
\end{equation}
The flow equations for the vector field on $\real^n$ given by $c^i X_i$, $c^i \in \real$,  are determined from \EqRef{LieX4} to be
\begin{equation}
\begin{aligned}
\frac{dx^s}{dt} & = c^r \left[ {\rm e}^{x^n[\widetilde \ad(\be_n)]}  \right]^u_r  \tilde \xi^s _ u, \\
\frac{dx^n}{dt} &= c^n.
\end{aligned}
\EqTag{Flow1}
\end{equation}
Integrating the last equation in \EqRef{Flow1} gives $x^n= c^n t+c^0$. Therefore the first equation in \EqRef{Flow1} can be written
\begin{equation}
\frac{dx^s}{dt}  = A^u(t)   \tilde \xi^s _ u, \quad s=1,\ldots,n-1,
\EqTag{Flow3}
\end{equation}
where 
$$
A^u(t) =  \left[{\rm e}^{(c^n t+c^0) [\widetilde \ad(\be_n]]}  \right]^u_r c^r
$$

Equation \EqRef{Flow3} is the equation of Lie-type for the function  $\sigma : \real \to \lieh$ given by
$$
\sigma(t)  =  A^u(t) \tilde X_u.
$$
Equations of Lie  type  admit global solutions in $t$ (\cite{onishchik:1993a} page 37), and hence $c^r X_r$ is complete. 
\end{proof}

In Section \StRef{multmap} we construct the multiplication map $\mu:G\times G \to G$, whose existence is guaranteed by Theorem \StRef{LG1}, using matrix exponentiation and quadrature.  In order to do this, we return to the question from the introduction of finding $\rho:M \to G$ in equation \EqRef{MF1} for
solvable Lie groups.

\subsection{Constructing the map $\rho:M\to G$  for solvable $G$}\StTag{movf}

Let $\{\omega^i\}_{1\leq i \leq n}$ be $n$ differential forms on a simply connected manifold $M$ satisfying equations \EqRef{FIE} (written as $d\bfomega_\lieg + \frac{1}{2} \bfomega_\lieg \wedge \bfomega_\lieg=0$ in Remark \StRef{BFR}). As noted in the introduction, it is well known that there exists a function $ \rho:M \to G $ such that $\omega^i= \rho^* \tau^i $
where $\tau^i$ are the left invariant Maurer Cartan forms on a Lie group $G$ with Lie algebra $\lieg$ (see \cite{sharpe:1997a}). Furthermore $\rho$ is unique up to left translation by an element of $G$. We now show how Corollary \StRef{C2} combined with Theorem \StRef{LG1} determines $\rho$ when $\lieg$ is solvable.

\begin{Theorem} \StTag{CMF}   Let $\lieg$ be an $n$-dimensional solvable real Lie algebra and let $\beta=\{ {\bf e}_i \}_{1 \leq i \leq n}$ be a basis for $\lieg$ adapted to a sequence of codimension one ideals  $\liek_s\subset \lieg$, $s=0,\ldots,n$. Let $(x^1,\ldots, x^n)$ be coordinates on the simplify connected solvable Lie group $G$ with basis of left invariant forms $\{ \tau^i \}_{1\leq i \leq n}$ given in \EqRef{Lieom}.

Let $\{\omega^i\}_{1\leq i \leq n}$ be $n$ differential forms on a simply connected manifold $M$ satisfying the structure equations \EqRef{FSE} where $C^i_{jk}$ are computed in the basis $\beta$ for $\lieg$ and let
$\rho : M \to G$ be given by 
\begin{equation}
\rho(p) = (x^1= f^1(p), \ldots,x^n= f^n(p) ), \quad {\rm for \ all} \ p \in M
\EqTag{deff1}
\end{equation}
where $f^i:M \to \real$ are the functions from Corollary \StRef{C2}.  Then 
\begin{equation}
\rho^* \tau^i = \omega^i.
\EqTag{pbf}
\end{equation}
\end{Theorem}

\begin{proof}  Let $f^i:M \to \real$ be the functions in Corollary \StRef{C2}  so that equation \EqRef{FIE} holds. Solving for $\omega^i$ in equation \EqRef{FIE} gives
\begin{equation}
\left[ \begin{array}{c} \omega^1\\ \vdots \\ \omega^n  \end{array} \right]
=
\rme^ {-f^n[\ad(\be_n)]} \, 
\left[ \begin{array}{cc}  \rme^{-f^{n-1}[\ad_{1}( \bfe_{n-1})]} & {\bf 0}_{1} \\ {\bf 0}_{1} ^T & 1 \end{array} \right]\cdots 
\left[ \begin{array}{cc}  \rme^{-f^{2}[\, \ad_{n-2}(\be_2)\,]} & {\bf 0}_{n-2} \\ {\bf 0}^T_{n-2} & {\rm I}_{n-2} \end{array} \right] 
\left[ \begin{array}{c}  d  f^1\\ \vdots \\df^n \end{array} \right].
\EqTag{MFP1}
\end{equation}
With $\{ \tau^i \}_{1\leq i \leq n}$ whose
formula is in given in  \EqRef{Lieom} in coordinates $(x^i)_{1\leq i \leq n}$, the map $\rho:M \to G$ in equation \EqRef{deff1} given in coordinates by  $x^i=f^i(p)$ for all $ p\in M$, satisfies equation \EqRef{pbf}. This follows at once by substituting $x^i=f^i$ in the expression for $\tau^i$ in equation \EqRef{Lieom}
which produces the expression for $\omega^i$ in equation \EqRef{MFP1}. 
\end{proof}

\begin{Remark} \StTag{RM1} Theorem \StRef{CMF} shows that determining the map $\rho:M \to G$ in equation \EqRef{MF1} for a solvable Lie group $G$  can be found by quadrature and the matrix exponential and is in fact found trivially from the $n$ functions  $f^i$ in Corollary \StRef{C2}. 
\end{Remark}

\begin{Example} In Example \StRef{Ex1} we have the three $1$-forms $\omega^1,\omega^2,\omega^3$ on $M^0$ given in equation \EqRef{omex1} which satisfy,
$$
d\omega^1=- \omega^2 \wedge \omega^3, \quad d \omega^2 = 0, \quad d \omega^3 =0.
$$
The $3$-dimensional simply connected Lie group  $G$ with Lie algebra $[\be_2,\be_3]=\be_1$
constructed as in Theorem \StRef{LG1} with coordinates $(x^1,x^2,x^3)$ has as a basis of left invariant forms (see \EqRef{Lieom}),
\begin{equation}
\tau^1=dx^1+x^3 dx^2, \quad \tau^2 = dx^2,\quad \tau^3=dx^3.
\EqTag{mte1}
\end{equation}
These $1$-forms satisfy $d \tau^1= -\tau^2 \wedge \tau^3$.

The functions $f^1,f^2$, and $f^3$ are given in equations \EqRef{om3} and \EqRef{ex1f1f2} in Example \StRef{Ex1} and these give the map $\rho:M^0 \to G$ in equation \EqRef{deff1}
 in Theorem \StRef{CMF} by
\begin{equation}
\begin{aligned}
x^1 &= f^1= x+ \frac{u_x^{10}+3u u_{xx}^2u_x^4-3 u u_{xx}^3}{ 3 u_ x u_{xx}^3},  \\
x^2 &= f^2= u+\frac{u_x^6}{2 u_{xx}^2},\qquad x^3 = f^3= \frac{1}{u_x}-\frac{u_x^3}{u_{xx}}.
\end{aligned}
\EqTag{GFI}
\end{equation}
It is easy to check $\rho^* \tau^i=\omega^i$, $i=1,2,3$, with $\omega^i $ in \EqRef{omex1} and $\tau^i$ in equation \EqRef{mte1}.

We now demonstrate that $\rho$ is equivariant.  Let $X_1,X_2,X_3$ be the right invariant vector fields on $G$ dual to $\tau^i$. Then $\rho_*(Z_i) = X_i$, and $\rho_* :\Gamma \to \lieg$ is an isomorphism, where $\lieg$ is the Lie algebra of right invariant vector fields and $\Gamma=\spn\{Z_i\}$  are the infinitesimal generators of the action  of $G$ on $M$. By an application of Theorem 13.1 in \cite{anderson-fels:2015a}, the function $\rho$ is $G$-equivariant.  

To explicitly show the equivariance of $\rho$, we make the change of coordinates $(a=x^1+x^2x^3,b=x^2,c=x^3)$ on $G$ so that the map $\rho:M^0 \to G$  from \EqRef{GFI} is given in these new coordinates by
\begin{equation}
a= x+\frac{u_x^5}{2u_{xx}^2}-\frac{u_x^9}{6u_{xx}^3}, \quad b = u+ \frac{u_x^6}{2 u_{xx}^2},\quad 
c=  \frac{1}{u_x}-\frac{u_x^3}{u_{xx}}.
\EqTag{newGFI}
\end{equation}
With the multiplication map on $G$ given by $(a,b,c)\cdot(a'b'c') = ( a+a'+cb',b+b',c+c')$ the (local) action of $G$ on $M^0$ is
\begin{equation}
\mu(a,b,c; x,u,u_x,u_{xx}) = \left(\ x+a+c u,\ u+b,\  \frac{u_x}{cu_x+1}, \frac{u_{xx}}{(cu_x+1)^3}\ \right)\, .
\EqTag{Arho}
\end{equation}
From equations \EqRef{Arho} and \EqRef{newGFI}  the equivariance condition $\rho(\mu(a,b,c; x,u,u_x,u_{xx}))=(a,b,c) \cdot \rho(x,u,u_x,u_{xx})$ is easily checked. 

Finally we point out that $\rho$ is the moving frame (see \cite{fels-olver:1999a}) for the action of $G$  on $M^0$ given in \EqRef{Arho}, that is determined from the cross-section $K= \rho^{-1}(0,0,0)$ to the orbits of $G$. Since $K$  corresponds to the level set of a complete set of first integrals, the cross-section $K$ is a solution to the ODE in \EqRef{621}. The solution is determined from $\rho^{-1}(0,0,0)$ using equation \EqRef{newGFI} and is (the prolongation of) $u= -\frac{1}{2}(3x)^{\frac{2}{3}}$. 

\end{Example}

\subsection{Constructing the multiplication function for simply connected solvable $G$.}\StTag{multmap}

The results from Section \StRef{Sreduce} along with Theorem \StRef{CMF} in Section \StRef{Lie3G} will now be used to construct the multiplication map $\mu$ whose existence is guaranteed in Theorem \StRef{LG1} using only $n$ quadratures and matrix exponentials. This is given by the following theorem.

\begin{Theorem}\StTag{L3M}  Let $\lieg$ be an $n$-dimensional solvable Lie algebra and let $\beta=\{ {\bf e}_i \}_{1 \leq i \leq n}$ be a basis for $\lieg$ adapted to a sequence of codimension one ideals $\liek_s\subset\liek_{s-1}$, $s=1,\ldots,n$. Let $\bftau=[ \tau^1,\ldots,\tau^n]^T$ be the $n$-vector of left invariant $1$-forms defined in equation \EqRef{Lieom} on the corresponding simply connected Lie group $G$ with identity chosen to be ${\bf 0}\in \real^n$ (Theorem \StRef{LG1}).

On $G \times G$ with coordinates $(x^1,\ldots,x^n,y^1,\ldots, y^n)$ define the $n$-vector differential $1$-forms $\bfomega=[\omega^1,\ldots,\omega^n]^T$ by
\begin{equation}
\bfomega = \rme^{-y^n [\ad ({\bf e_n})]}\cdots  \rme^{-y^1 [\ad( {\bf e_1})] }  
 \pi_1^* \bftau + \pi_2^*\bftau
\EqTag{htheta}
\end{equation}
where  $\pi_1(x,y)=x$ and $\pi_2(x,y)=y$ are the two projection maps on $G\times G$. 

\noindent
{\bf [i]} The forms $\omega^i$ satisfy
\begin{equation}
d \omega^i + \frac{1}{2} C^i_{jk} \omega^j \wedge \omega^k
\EqTag{dbfom}
\end{equation}
where $C^i_{jk}$ are the structure constants of $\lieg$ in the basis $\beta$.

\noindent 
{\bf [ii]} Let $\mu:G \times G \to G$ with $\mu({\bf 0},{\bf 0})={\bf 0}$  be the map constructed in Theorem \StRef{CMF},  so that $\omega^i=\mu^*\tau^i$. Then $\mu:G\times G\to G$ is the multiplication map for the Lie group $G$ with basis of left invariant forms $\{ \tau^i \}_{1\leq i \leq n} $ and identity ${\bf 0} \in \real^n$.
\end{Theorem}

Part {\bf ii]} in Theorem \StRef{L3M} states that if $(z^i)$ denote coordinates on the target $G$ with $\tau^i$ in equation \EqRef{Lieom} given in the $z^i$ coordinates,  and $f^i(x,y)$ the functions determined by applying Corollary \StRef{C2} to $\bfomega$ in equation $\EqRef{htheta}$, then the multiplication map is given explicitly by $z^i =f^i(x,y)$. This is demonstrated in Example \StRef{Ex3} below. The choice of identity element ${\bf 0}$ in Theorem \StRef{L3M}  was made for simplicity.



In order to prove Theorem \StRef{L3M} we first present a way to construct the multiplication map $\mu:G\times G \to G$ using integral manifolds. The proof of the following theorem is given in Section 7 of \cite{gardner:1989a}.

\begin{Theorem}\StTag{IntMG} Let $G$ be a simply connected Lie group and let $\{\tau^i\}_{1\leq i \leq n}$ be a basis of left invariant $1$-forms on $G$. On $G \times G$  define the rank $n$ Pfaffian system $I =\spn \{\theta^1,\ldots,\theta^n \}$ where
\begin{equation}
\theta^i =\pi_2^* \tau^i-\pi_1^*\tau^i.
\EqTag{CTI}
\end{equation}
The Pfaffian system $I$ is completely integrable. If $(x_0,y_0) \in G\times G$ and $\CalL_{(x_0,y_0)}$ is the maximal integral manifold through $(x_0,y_0)$, then
\begin{equation}
\CalL_{(x_0,y_0)}= \{\ ( x,   L_{y_0x_0^{-1}} x) , \quad x\in G\ \},
\EqTag{prem}
\end{equation}
where $L_{y_0 x_0^{-1}}$ is left multiplication by $\mu(y_0, x_0^{-1})$. 
\end{Theorem}


We now determine a  set of generators for $I=\spn\{\theta^1,\ldots,\theta^n\}$, with $\theta^i$ given in equation \EqRef{CTI}, which will allow us to determine the maximal integral manifolds of $I$ by quadratures and matrix exponentiation. This will determine the multiplication map using equation \EqRef{prem} (see Lemma \StRef{Pmu}).  The details are contained in the next two lemmas.

\begin{Lemma}\StTag{PreAd} Let $\lieg$ be an $n$-dimensional solvable Lie algebra and let $\beta=\{ {\bf e}_i \}_{1 \leq i \leq n}$ be a basis for $\lieg$ adapted to a sequence of codimension one ideals $\liek_s\subset\liek_{s-1}$, $s=1,\ldots,n$. Let $\bftau=[ \tau^1,\ldots,\tau^n ]^T$ be the vector of left invariant $1$-forms on $G$  defined in equation \EqRef{Lieom}.  Let  $\tilde \theta^i$ be the $n$ 1-forms on $G\times G$ given by
\begin{equation}
\tilde {\bftheta} = \rme^{x^1 [\ad(\be_1)]} \cdots \rme^{x^n [\ad(\be_n)]}
\left(\pi_2^* \bftau-\pi_1^*\bftau  \right).
\EqTag{bftf2}
\end{equation}
These form a basis for the Pfaffian system $I$ in Theorem \StRef{IntMG}, $I =\spn\{ \tilde \theta^1,\dots, \tilde \theta^n \}$ where the structure equations for $\tilde \theta^i$ are
\begin{equation}
d \tilde \theta^i + \frac{1}{2} C^i_{jk} \tilde \theta^j \wedge \tilde \theta^k =0.
\EqTag{dtt0}
\end{equation}
\end{Lemma}

The generators $\bfttheta$ for $I$ in \EqRef{bftf2} and Theorem \StRef{CMF} can now be used to construct the multiplication map.

\begin{Lemma}\StTag{Pmu} Suppose the hypothesis of Lemma \StRef{PreAd} on $\lieg$, $\beta$ and $\bftau$ are satisfied and let $\tilde \theta^i$ be the  $1$-forms on $G\times G$ from equation \EqRef{bftf2} which satisfy \EqRef{dtt0}. Let $\rho: G\times G \to G$ be the unique map constructed using Theorem \StRef{CMF} so that
$\rho({\bf 0}, {\bf 0}) = {\bf 0}$ and
$$
\rho^*\tau^i = \tilde \theta^i.
$$
Then
\begin{equation}
\rho(x,y) = \mu(y, x^{-1})
\EqTag{Fxy}
\end{equation}
where $\mu:G\times G\to G$ is the multiplication map with ${\bf 0}\in G$ as the identity. 
\end{Lemma}


We now prove Lemmas \StRef{PreAd} and \StRef{Pmu}.

\begin{proof} (Lemma \StRef{PreAd}) On $G\times G$ consider the vector of $2n$ differential forms given by 
\begin{equation}
\bfOmega=[ \theta^1, \ldots, \theta^n , \pi_1^* \tau^1, \ldots \pi_1^* \tau^n ]^T=\left[ \begin{array}{c} \bftheta \\ \bftau_1 \end{array} \right]
\EqTag{dOm}
\end{equation}
where from equation  \EqRef{CTI} $\theta^i= \pi_2^* \tau^i-\pi_1^*\tau^i$ and $\bftau_1 = \pi_1^* \bftau $.
 The structure equations for $\bfOmega$ follow from equation \EqRef{LTF} and are
\begin{equation}
\begin{aligned}
d \theta ^i +\frac{1}{2}C^i_{jk} \theta^j \wedge \theta^k + C^i_{jk} \theta^j \wedge \tau_1^k & =0, \\
d \tau_1^i  +\frac{1}{2}C^i_{jk} \tau_1^j \wedge \tau_1^k & = 0.
\end{aligned}
\EqTag{SE1}
\end{equation}
The $2n$ forms in \EqRef{dOm} are a basis of left invariant forms on $G \times G$ which are dual to the basis for  $\lieg \times \lieg$ of left invariant vector fields given by
\begin{equation}
\gamma= \{ \ Y_1,\ldots, Y_n,  X_1+Y_1, X_2+Y_2,\ldots , X_n+Y_n \ \}
\EqTag{BCD1}
\end{equation}
where $\{X_1,\ldots, X_n,Y_1\ldots,Y_n\}$ is the dual frame to the coframe $\{ \pi_1^* \tau^1,\ldots, \pi_1^* \tau^n, \pi_2^* \tau^1, \ldots, \pi_2^* \tau^n \}$  on $G\times G$.

Define subalgebras $\liel_{s} \subset \lieg$, $s=0,\ldots, 2n$ by
\begin{equation}
\liel_{a}= \spn\{\ Y_1,\ldots, Y_n, X_1+Y_1, X_2+Y_2,\ldots , X_{n-a}+Y_{n-a}\ \}, \quad 0\leq a \leq n-1,
\EqTag{firstl}
\end{equation}
and
\begin{equation}
\liel_{n+a} = \spn\{\ Y_1,\, X_2,\ldots ,\, Y_{n-a} \ \}, \quad  0 \leq a \leq n.
\EqTag{secondl}
\end{equation}
It is then simple to show that $\liel_s\subset \liel_{s-1}$, $s=1,\ldots,2n$ forms a sequence of codimension one ideals for the solvable Lie algebra $\lieg \times \lieg$. Furthermore, by definition, the basis \EqRef{BCD1} is adapted to the sequence of codimension one ideals $\liel_s\subset \liel_{s-1}, s=1\ldots 2n$.

We now apply Corollary \StRef{C1} using the first $n$  terms in the sequence, $\liel_{a}, a=0,\ldots, n$. This will result in a change in the coframe $[ \bftheta, \bftau_1 ]^T$ on $G\times G$ where the terms $C^i_{jk} \theta^j \wedge \tau^k_1$ on the first line of equation \EqRef{SE1} don't appear (see the paragraph following the proof of Theorem \StRef{reduce}). 

In order to apply Corollary \StRef{C1} we need to compute 
$$
\ad_a (X_{n-a} +Y_{n-a}) : \liel_{a} \to \liel_{a}, \quad 0\leq a \leq n
$$
in terms of the basis \EqRef{BCD1}. This is easily done giving
$$
[ \ad_a (X_{n-a} +Y_{n-a}) ]=\left[ \begin{array}{cc} [\ad(\be_{n-a})] & 0 \\ 0 & [\ad_a(\be_{n-a})] \end{array} \right]
$$
where $[\ad_a(\be_{n-a})]$ is the $n-a$ by $n-a$ 
matrix representation of the derivation defined in equation \EqRef{alseq} in the basis $\beta=\{\be_i\}_{1\leq i \leq n}$ for $\lieg$, and we have used the isomorphism \EqRef{giso}.

By using $\Omega^{2n} = dx^n$ in equation \EqRef{dOm}  we determine the right most matrix in equation \EqRef{seqreduce} by using $f^n= x^n$. The first reduction (which is equivalent to one application of Theorem \StRef{reduce}) is then given by
\begin{equation}
\hat \bfOmega =  
\left[ \begin{array}{cc} \rme^{x^n[\ad(\be_n)]} & {\bf 0}_n  \\ {\bf 0}_n^T & \rme^{x^n [\ad(\be_n)]} \end{array} \right]
 \left[ \begin{array}{c} \bftheta \\ \bftau_1  \end{array} \right].
\EqTag{bfOm1}
\end{equation}
Now using equation \EqRef{Lieom} and the fact that $\rme^{x^n[\ad(\be_n)]} \rme^{-x^n[\ad(\be_n)]}  = {\rm I}_{n}$ equation \EqRef{bfOm1} simplifies to
\begin{equation}
\hat \bfOmega =  
\left[ \begin{array}{c} \rme^{x^n[\ad(\be_n)]} \bftheta  \\ A_{n-1}(x^{n-1}) \ldots A_{2}(x^2) {\bf dx}.
\end{array} \right]
\EqTag{bfOm2}
\end{equation}
where
\begin{equation}
A_{n-a}(x^{n-a}) = \left[ \begin{array}{cc} \rme^{-x^{n-a}[\ad_a(\be_{n-a})]} & {\bf 0}_a \\ {\bf 0}^T_a & {\rm I}_{a} \end{array} \right], \quad a=0,\ldots, n-2.
\EqTag{PM}
\end{equation}
are the matrices in equation \EqRef{Lieom} and \EqRef{bfOm2}.

Continuing by induction we have that $f^i = x^i$ and that after $n$ steps, Corollary \StRef{C1} produces
$$
\left[ \begin{array}{cc} \rme^{x^1[\ad(\be_1)]} & {\bf 0}  \\ {\bf 0} ^T & {\rm I}_n \end{array} \right]
\left[ \begin{array}{cc} \rme^{x^2[\ad(\be_2)]} & {\bf 0}   \\ {\bf 0} ^T & A_2(-x^2)  \end{array}\right]\cdots
\left[ \begin{array}{cc} \rme^{x^n[\ad(\be_n)]} & {\bf 0}  \\ {\bf 0} ^T & A_n(-x^n) \end{array} \right]
\left[ \begin{array}{c} \bftheta \\ \bftau  \end{array} \right] =  \left[ \begin{array}{c} {\tilde \bftheta} \\ {\bf dx}  \end{array} \right].
$$
where again $A_s(x^s)$ are given in \EqRef{PM} and $\tilde \bftheta$ are the sought after forms in equation \EqRef{bftf2}. By construction, the structure equations  \EqRef{dtt0} hold which proves the lemma.
\end{proof}


We now turn to the proof of Lemma \StRef{Pmu}.

\begin{proof} (Lemma \StRef{Pmu})  Theorem \StRef{CMF} applies to ${\bfttheta}$ in equation \EqRef{bftf2}, and so let $\rho:G\times G \to G$ be the unique function satisfying,
$$
\rho^* \tau^i =\tilde \theta^i,
$$
and $\rho({\bf 0},{\bf 0})={\bf 0}$, where ${\bf 0}$ is the identity in $G$.
Let $(x_0,y_0)\in G \times G$, $z_0= \rho(x_0,y_0)$, and let 
\begin{equation}
\CalN = \{ (x,y) \in G \times G \ | \ \rho(x,y) =z_0 \}= \rho^{-1}(z_0).
\EqTag{calN}
\end{equation}
The functions $\rho^1,\ldots,\rho^n$ are $n$ functionally independent first integral of $I$ and so
the embedded manifold $\CalN$ is an integral manifold of $I$ of dimension $n$ which contains $(x_0,y_0)$.  

Let  $\CalL$ be the maximal integral manifold of $I$ through $(x_0,y_0)$ (Theorem \StRef{IntMG}). By Theorem \StRef{IntMG} shows that the projection map $\pi_1 : G \times G$ restricted to $\CalL$ is a diffeomorphism, while by the connectivity of $\CalL$ implies $\CalL \subset \CalN$. Therefore $\pi_1:\CalN \to G$ is a covering map and hence a diffeomorphism and consequently $\CalN = \CalL$. A similar argument shows that $\pi_2:\CalN \to G$ is a diffeomorphism.

The integral manifold $\CalL$ (and hence $\CalN$ in equation \EqRef{calN}) is given by the graph in \EqRef{prem}. Therefore given any $x\in G$ we can solve the equations $ \rho(x,y)=z_0$  for $y$   giving
\begin{equation}
y = g(z_0, x) )= \mu(y_0 x_0^{-1}, x).
\EqTag{detz}
\end{equation}
Equation \EqRef{detz} implies $ z_0 = \mu( y_0, x_0^{-1}) $. Since $(x_0,y_0)$ was arbitrary this proves the theorem.
\end{proof}

\begin{Corollary}\StTag{CC1} Let $\lieg$ be a solvable Lie algebra with basis $\beta=\{ \be_i\}_{1\leq i\leq n}$ adapted to a sequence of codimension one ideals $\liek_s\subset \liek_{s-1},\ s=1,\ldots,n$, and let $G$ be the Lie group in Theorem \StRef{L3M} with basis of left-invariant forms $\{ \tau^i \}_{1\leq i \leq n}$ given in equation \EqRef{Lieom} with coordinates $(x^i)_{1\leq i \leq n}$ for $G$. Then 
\begin{equation}
[\Ad(x)] =   \rme^{x^1 [\ad ({\bf e_1})] }\ldots\rme^{x^n [\ad( {\bf e_n})] }
\EqTag{Adx}
\end{equation}
where $[\Ad(x)]$ is the matrix representation of $\Ad(x),\ x \in G$ in the basis $\beta$. 
\end{Corollary}

\begin{proof}  Let  $\rho :G \times G \to G$ be $\rho(x,y) = \mu(y, x^{-1})$. In components using the basis $\beta$ we have
\begin{equation}
\rho^* {\bftau} = [Ad(x)] \pi_2^* \bftau- \pi_1^* \bftau_R,
\EqTag{pbad}
\end{equation}
where $\bftau_R$ are the right invariant $1$-forms on $G$ which are equal to the left invariant $1$-forms $\bftau$ at the identity ${\bf 0}$. 
From Lemma \StRef{Pmu} we also have  $\rho^* \bftau = \tilde \bftheta $. Using the expression in equation \EqRef{bftf2} for $\tilde \bftheta$ and equation \EqRef{pbad} gives 
$$
\rme^{x^1 [\ad(\be_1)]} \cdots \rme^{x^n [\ad(\be_n)]}
\left(\pi_2^* \bftau-\pi_1^*\bftau  \right)=[Ad(x)] \pi_2^* \bftau- \pi_1^* \bftau_R
$$
leading to the expression for $[\Ad(x)]$ in equation \EqRef{Adx}.
\end{proof}

We now prove Theorem \StRef{L3M}. 
\begin{proof} (Theorem \StRef{L3M}) The proof relies on equation \EqRef{Adx} in  Corollary \StRef{CC1} written as
$$
[\Ad(\bfy^{-1})] = \rme^{-y^n [\ad( {\bf e_n})]}\cdots  \rme^{-y^1 [\ad( {\bf e_1})] }.
$$
Therefore equation \EqRef{htheta} can be written
\begin{equation}
\bfomega =  [\Ad(\bfy^{-1})] \pi_1^* \bftau + \pi_2^*\bftau.
\EqTag{htheta2}
\end{equation}
Proposition 4.10 in \cite{sharpe:1997a} and equation \EqRef{htheta2} shows that 
\begin{equation}
\omega^i =\mu^*\tau^i
\EqTag{mpb}
\end{equation} 
where $\mu:G\times G \to G$ is the multiplication map. By taking the exterior derivative of equation \EqRef{mpb} we find  $\bfomega$ in \EqRef{htheta} satisfy the structure equations \EqRef{dbfom}.  This proves part {\bf i]}.

To prove part {\bf [ii]}, let $\rho:G\times G\to G$ be constructed by Theorem \StRef{CMF}, where $f^i$ are
chosen to satisfy $f^i({\bf 0},{\bf 0})= 0$. Theorem 8.7 of \cite{sharpe:1997a} shows that the multiplication function, $\mu:G\times G \to G$, is the unique map satisfying $\mu^* \tau^i = \omega^i$ and $\mu({\bf 0},{\bf 0}) = {\bf 0}$. Therefore the function $\rho=\mu$  is the multiplication map.
\end{proof}

\begin{Remark}\StTag{RightF}
If instead of the forms $\bfomega$ in equation \EqRef{htheta} of Theorem \StRef{L3M} we take
\begin{equation}
\bfomega = 
 \pi_1^* \bftau +  \rme^{-x^n [\ad ({\bf e_n})]}\cdots  \rme^{-x^1 [\ad( {\bf e_1})] } \pi_2^*\bftau
\EqTag{hthetar}
\end{equation}
in the construction of $\rho:G\times G \to G$ with $\rho^*\tau^i=\omega^i$, then $\rho$ is the multiplication map where the forms $\tau^i$ are {\it right} invariant. 
\end{Remark}

\begin{Example} \StTag{Ex3} We continue with Example \StRef{Ex2} and produce the multiplication map using Theorem \StRef{L3M} for the corresponding Lie groups where according to Theorem \StRef{LG1}, $\tau^i, i=1,\ldots, 5$ in equation \EqRef{Ex1forms} form a basis for the left invariant forms. First we need the forms $\omega^i, \ i=1,\ldots,5$ in equation \EqRef{htheta} in Theorem \StRef{L3M}. We compute the matrix $[\Ad(\bfy^{-1})]$ in equation \EqRef{htheta}  using the structure constants in equation \EqRef{SCe1} (or see equation \EqRef{htheta2}) to be
\begin{equation}
\begin{aligned}
&&& [\Ad(\bfy^{-1})]= \rme^{-y^5 [\ad ({\bf e_5})]}\cdots  \rme^{-y^1 [\ad( {\bf e_1})] } = \qquad \qquad \qquad \qquad \qquad \qquad \qquad \qquad \qquad \qquad \\
&&& \left[
\begin{array}{ccccc}{\rme}^{ a\, y^5 +b\, y^4 } & 0&0&-b \, y^1 {\rme}^{a\, y^5 + b\, y^4 }&-a\, y^1 {\rme}^{a\, y^5 a+b\, y^4}\\
0&{\rme}^{y^4} \cos y^5  &{\rme}^{y^4} \sin y^5 &-{\rme}^{y^4} (y^2 \cos y^5+y^3 \sin y^5 )&{\rme}^{y^4} (y^2\sin y^5-y^3 \cos y^5  )\\
0&-{\rme}^{y^4}\sin y^5 &{\rme}^{y^4} \cos y^5 &{\rme}^{y^4} (y^2 \sin y^5-y^3 \cos y^5 )&{\rme}^{y^4} (y^2 \cos y^5+y^3 \sin y^5 )\\
0&0&0&1&0 \\
0&0&0&0&1 
\end{array} \right]
\end{aligned}
\EqTag{CFE}
\end{equation}
We multiply the vector of forms in equation \EqRef{Ex1forms} by the matrix in \EqRef{CFE} to produce the forms ${\bfomega}$ in equation \EqRef{htheta} 
\begin{equation}
\begin{aligned}
\omega^1& =\rme^{a\, y^5+ b \, y^4} \left( d y^1 +  \rme^{a\, x^5+ b \, x^4}dx^1  - a\,  y^1 d x^5 - b y^1 dx^4\right)\\
\omega^ 2& = \rme^{y^4}\left(  \rme^{x^4} ( \cos(x^5+y^5) dx^2+ \sin(x^5+y^5)dx^3) +\cos y^5  dy^2+\sin y^5 dy^3\right.\\ 
& \left. \qquad \qquad -(y^2\cos y^5 +y^3 \sin y^5)dx^4+
(y^2 \sin y^5 -y^3\cos y^5)dx^5 \right)\\
\omega ^3& =  \rme^{y^4}\left(  \rme^{x^4} ( \cos(x^5+y^5) dx^3-\sin(x^5+y^5)dx^2) +\cos y^5  dy^3-\sin y^5 dy^2\right.\\ 
& \left. \qquad \qquad 
(y^2\sin y^5 -y^3 \cos y^5)dx^4+
(y^2 \cos y^5 +y^3\sin y^5)dx^5 \right)\\
\omega^ 4& =dx^4+dy^4 \\
\omega^5& = dx^5 + dy^5
\end{aligned}
\EqTag{MCEF} 
\end{equation}
We now find, by quadratures, the functions in Corollary \StRef{C2} which satisfy $f^i({\bf 0},{\bf 0})= 0$. 

We find $f^4, f^5$ from $d\omega^4=0$ and $d\omega^5=0$ in the last two lines of equation \EqRef{MCEF} to be,
\begin{equation}
\begin{aligned}
f^5&=x^5 + y^5\\
f^4&=x^4 + y^4.
\end{aligned}
\EqTag{Ef45}
\end{equation}
Reducing  $\bfomega$ by using Theorem \StRef{reduce} two times applied to equation \EqRef{MCEF} we get,
\begin{equation}
\begin{aligned}
\hbfomega & = \left[ \begin{array}{cc} \rme^{(x^4+y^4 )[\ad_1(\be_4)]} & {\bf 0}_1 \\ {\bf 0}_1^T & 1 \end{array} \right] \rme^{(x^5+y^5) [\ad(\be_5)]} \bfomega \\
& =   \left[ \begin{array}{ccccc} \rme^{-b (x^4+y^4)} & 0 &  0 &0 & 0  \\ 0 & \rme^{-(x^4+y^4)} & 0 &0 & 0  \\ 0 & 0 & \rme^{-(x^4+y^4)} &0 & 0   \\ 0 & 0& 0 & 1 & 0 \\ 0 & 0 & 0 & 0& 1   \end{array} \right]
\left[ \begin{array}{ccccc} \rme^{-a (x^5+y^5)} & 0 &  0 & 0 & 0 \\ 0 & \cos (x^5+y^5)  &- \sin (x^5+y^5)  & 0& 0 \\ 0 & \sin (x^5+y^5)  & \cos (x^5+y^5) & 0  & 0\\ 0 & 0  &0 & 1 & 0\\ 0 & 0 & 0& 0 & 1\end{array} \right]
 \left[  \begin{array}{c} \omega^1 \\ \omega^2 \\ \omega^3 \\ \omega^4 \\ \omega^5 \end{array} \right] \\
& = \left[ \begin{array}{c} 
dx^1+ \rme^{-a \, x^5-b\, x^4}\left(dy^1 -b\, y^1 dx^4-a \, y^1 dx^5\right) \\ 
dx^2+\rme^{-x^4}\left(\cos x^5 dy^2 -\sin x^5 dy^3-(y^2\cos x^5 -y^3\sin x^5 )dx^4-(y^3\cos x^5+y^2\sin x^5)dx^5\right) \\ 
dx^3+\rme^{-x^4}\left(\sin x^5 dy^2 +\cos x^5 dy^3-(y^3\cos x^5+y^2\sin x^5 )dx^4+(y^2\cos x^5-y^3\sin x^5 )dx^5\right) \\ 
dx^4+dy^4\\ 
dx^5+dy^5 \end{array} \right]
\end{aligned}
\EqTag{domh}
\end{equation}
where we have used  the matrix exponentials computed in equation \EqRef{Adms}.
Since $[\ad_3(\be_2)]$ $[\ad_2(\be_3)]$ are zero matrices equation there is no more reduction to be done in equation \EqRef{FIE},  and $d \homega^3=0, d\homega^2=0$ and $d \homega^1=0$ in equation \EqRef{domh}. This produces the final three functions ($\hat \omega^i=df^i,  f^i({\bf 0},{\bf 0})=0$) by quadratures,
\begin{equation}
\begin{aligned}
f^3 &=x^3 + \rme^{-x^4} (y^2 \sin x^5 +  y^3 \cos x^5), \\
f^2 & =x^2 + \rme^{-x^4}( y^2 \cos x^5 -y^3 \sin x^5 ),\\
f^1 &= x^1 + \rme^{-a\, x^5  - b\, x^4 } y^1 .
\end{aligned}
\EqTag{Ef123}
\end{equation}
In accordance with Theorem \StRef{L3M} the multiplication map $z^i=f^i(x,y)$ is given using equations \EqRef{Ef45} and \EqRef{Ef123} by
\begin{equation}
\begin{aligned}
z^1 &= x^1 + \rme^{-a\, x^5  - b\, x^4 } y^1 ,\\
z^2 & =x^2 +\rme^{-x^4}( y^2 \cos x^5 -y^3 \sin x^5 ),\\
z^3 &=x^3 + \rme^{-x^4} (y^2 \sin x^5 +  y^3 \cos x^5), \\
z^4&=x^4 + y^4,\\
z^5&=x^5 + y^5.
\end{aligned}
\EqTag{MM}
\end{equation}
Equation \EqRef{MM} is the multiplication map ${ z}=\mu({ x}, {y})$ for the simply connected $5$-dimensional Lie group $G$ with basis of left invariant forms in \EqRef{MCEF} and ${\bf 0}$ as the identity.
\end{Example}

\section{Conclusion}   It seems possible that an explicit formula for the multiplication map for a simply connected solvable Lie group can be given in terms of the structure constants of the Lie algebra using the techniques developed in this paper, but the author has not been able to find it.


\begin{bibdiv}
\begin{biblist}


\bib{anderson-fels:2005a}{article}{
  author={Anderson, I. M.},
  author={Fels, M. E.},
  title={Exterior Differential Systems with Symmetry},
  journal={Acta. Appl. Math.},
  year={2005},
  volume={87},
  pages={3--31},
}

\bib{anderson-fels:2015a}{article}{
  author={Anderson, I. M.},
  author={Fels, M. E.},
  title={B\"acklund transformations for Darboux integrable
differential systems},
  journal={Sel.Math.New Ser.},
  year={2015},
  volume={21},
  pages={379-448},
}






\bib{doubrov:2000a}{article}{
author={B.M. Doubrov, B.P. Komrakov},
title={The constructive equivalence problem in differential geometry},
journal={Mat. Sb}, 
volume={191},
year={2000}, 
pages={655--681}
}

\bib{fels:2007a}{article}{
  author = {Fels, M. E.}, 	
  title={Integrating ordinary differential equations with symmetry revisited},
  journal ={Foundations of Computational Math.},
  volume = {7},
  year = {2007},
  pages = {417-454}
}

\bib{fels:2008a}{article}{
  author={Fels, M. E.},
  title={Exterior Differential Systems with Symmetry},
  journal={The IMA Volumes in Mathematics and its Applications},
  year={2008},
  volume={144, II},
  pages={351--366},
  }

\bib{fels-olver:1999a}{article}{  
author={M.Fels},author={P.J.Olver}, 
title={Moving Coframes II. Regularization and theoretical foundations}, 
journal={Acta. Appl. Math.}, 
volume={55},
year={1999},
pages={127--208}
}

\bib{flanders:1989a}{book}{
author={H. Flanders},
title={Differential Forms with Applications to the Physical Sciences},
publisher = {Dover},
year={1963,1989}
}

\bib{gardner:1989a}{book}{
author={R.B. Gardner},
title={The Method of Equivalence and is Applications},
publisher = {SIAM},
year={1989}
}

\bib{griffiths:1974a}{article}{
  author={Griffiths, P.},
  title={On Cartan's method of Lie groups and moving frames as applied to uniqueness and existence questions in differential geometry},
  journal={Duke Math. J.},
  year={1974},
  volume={41},
  pages={775--814}
}

\bib{jacobson:1962a}{book}{
author={N. Jacobson},
title={Lie Algebras},
publisher = {Dover},
year={1962}
}


\bib{lychagin:1991a}{article}{  
author={S.V. Duzhin},author={V.V. Lychagin}, 
title={Symmetries of Distributions and quadrature of ordinary
differential equations}, 
journal={Acta. Appl. Math.}, 
volume={24},
year={1991},
pages={29--57}
}



\bib{olver:1998a}{book}{
author={Olver, P.J.},
title={Applications of Lie Groups to Differential Equations},
publisher={Springer-Verlag},
year={1998},
}


\bib{onishchik:1993a}{book}{
  author = {A.L. Onishchik (Ed.)},
   title = {Lie groups, Lie algebras I},
 publisher = {Springer-Verlag},
  year={1993},
}




\bib{sharpe:1997a}{book}{
  author = {Sharpe, R.W.},
   title = {Differential Geometry: Cartan's Generalization of Klein's Erlangen Program},
 publisher = {Springer},
  year={1997},
}


\bib{vanest:1988a}{article}{
author={W. Van Est}, 
title={Une d\'emonstration de \'E. Cartan du troisei\`eme th\'eor\`eme de Lie},
journal={Travaux en Cours [Works in Progress], Hermann, Paris },
volume={27}, 
year={1988},pages={83--96}
}




\bib{winternitz:1976a}{article}{
author={J. Patera}, author={R. T. Sharp}, author={P. Winternitz}, author={H. Zassenhaus},
title={Invariants of real low dimensional Lie algebras},
journal={Journal of Math. Phys.},
volume={17}, number={6},
year={1976},pages={966--994}
}


\end{biblist}
\end{bibdiv}
\end{document}